 \documentclass[11pt]{article}
\usepackage[paper=a4paper,textwidth=162mm, textheight=238mm]{geometry}

\usepackage{graphicx}

\usepackage{amssymb}
\usepackage{amsthm}
\usepackage{amsmath}
\usepackage{amsfonts}
\usepackage{dsfont}

\usepackage{lineno}

\usepackage{cancel}
\usepackage{ulem}

\newtheorem{Theorem}{Theorem} 
\newtheorem{Definition}{Definition} 
\newtheorem{Proposition}{Proposition} 
\newtheorem{Conjecture}{Conjecture} 
\newtheorem{Assumption}{Assumption} 
\newtheorem{Lemma}{Lemma} 
\theoremstyle{remark}
\newtheorem{Remark}{Remark} 

\def \be{\begin{equation}}
\def \ee{\end{equation}}
\def\l{\left}
\def\r{\right}
\def\ba{\begin{array}}
\def\ea{\end{array}}
\def \And{\;\mbox{ and }\;}
\def \Pas{\mathbb{P}-\mbox{a.s.}}
\def \Pae{\mathbb{P}-\mbox{a.e.}}
\def \we{\tilde}
\def \wt{\hat}

\def \demi{\frac{1}{2}}
\def \eps{\varepsilon}
\def \vp{\varphi}

\def \1{{\bf 1}}
\def \ind#1{\mathds{1}_{\left\{#1\right\}}}
\def \Esp#1{\mathbb{E}\left[#1\right]}

\def \Pro#1{\mathbb{P}\left[{#1}\right]}
\def \cadlag{c\`adl\`ag }

\def \R{\mathbb{R}}
\def \Cc{{\cal C}}
\def \x{\times}
\def \Fc{{\cal F}}
\def \Lc{{\cal L}}
\def \P{\mathbb{P}}
\def \N{\mathbb{N}}
\def \F{\mathbb{F}}
\def \pourtout{\mbox{ for all } }
\def \Prod{\displaystyle\prod}
\def \Sc{{\cal S}}

\usepackage{authblk}
\title{Orbits in a
stochastic Goodwin-Lotka-Volterra model}
\author[1]{Adrien Nguyen Huu\thanks{Corresponding author: adrien.nguyenhuu@gmail.com}}
\author[2]{Bernardo Costa-Lima}
\affil[1]{{\small{ \it IMPA, Estrada Dona Castorina 110, Rio de Janeiro 22460-320, Brasil}}}
\affil[2]{{\small{\it McMaster University, Hamilton ON L8S 4L8, Canada}}}

\begin{document}
\maketitle




\begin{abstract}
This paper examines the cycling behavior of
a deterministic and a stochastic version
of the economic interpretation of 
the Lotka-Volterra model, the Goodwin model.
We provide a characterization
of orbits in the deterministic
highly non-linear model.
We then study the cycling behavior
for a stochastic version,
where a Brownian noise is introduced
via an heterogeneous productivity factor.
Sufficient conditions for existence of the system
are provided.
We prove that
the system produces cycles around an 
equilibrium point in finite time
for general volatility levels, using 
stochastic Lyapunov techniques for recurent domains.
Numerical insights are provided.
\end{abstract}

\begin{minipage}{0.9\textwidth}
\end{minipage}

\section{Introduction}

The Lotka-Volterra
equation is at the
heart of population dynamics,
but also possesses
a famous economic interpretation.
Introduced by Richard Goodwin \cite{goodwin1967}
in 1967,
the model in its modern form \cite{desai2006}
reduces to the planar oscillator on a subset $D$ of $\R_+$:

\be
\label{eq:goodwin_model}
\left\{
\begin{array}{ll}
	d x_t &= x_t  \l(\Phi(y_t)-\alpha\r)dt  \\
	d y_t &= y_t \l(\kappa(x_t) -\gamma\r)dt 
\end{array}
\right. \; ,
\ee
where $x_t$ denotes the wage share of the working population
and $y_t$ the employment rate,
$\alpha$ and $\gamma$ are constant, and
the following assumption is made on $\kappa$ and $\Phi$.
\begin{Assumption}
 \label{ass:non-linear goodwin}
 Consider system \eqref{eq:goodwin_model}.
 \begin{enumerate}
	\item[(i)] $\Phi\in\Cc^2([0,1))$, 
	$\Phi'(y)> 0$, $\Phi^{''}(y)\ge 0$
	for all $y\in [0,1)$,
	$\Phi(0)< \alpha$ and
	$\lim_{y \to 1^-} \Phi(y) = +\infty$.
	\item[(ii)] $\kappa\in\Cc^2(\R_+)$,
	$-\infty<\kappa'(x)<0$ for all $x\in \R_+$,
	$\kappa(0)>\gamma$ and
	$\lim_{x \to +\infty} \kappa(x) = -\infty$.
\end{enumerate}
\end{Assumption}

Lemma \ref{lem:goodwin system} below asserts that
Assumption \ref{ass:non-linear goodwin} is sufficient 
to have $(x_t, y_t)\in D:=\R^*_+\x(0,1)$ for any $t\ge 0$
if $(x_0, y_0)\in D$.
This property preserves the above interpretation 
for $x$ and $y$:
the employment rate cannot exceed one
for obvious reasons,
but the wage share can, depending on
the chosen economic assumptions, see \cite{grasselli2012}.
This distinctive feature of the economic version \eqref{eq:goodwin_model}
on its biological counterpart
follows from a construction based on assumptions
describing a closed capitalist economy.
It can be done in three steps:
\begin{enumerate}
 \item[(I)] Assume a Leontief
  production function $P_t = \min(K_t / \nu ; a_t y_t N_t)$
  with   full utilization of capital, i.e., 
  $K_t / \nu = a_t y_t N_t$.
  Here, $P_t$ is the yearly output,
  $K_t$ the invested capital,
  $\nu>0$ a capital-to-output ratio,
  $a_t:= a_0 \exp(\alpha t)$ is the average productivity of workers
  and $N_t:=N_0 \exp(\beta t)$ is the size of the labor class.
 \item[(II)] The capital depreciates and receives investment, i.e.,
  $
    dK_t/dt = (\kappa(x_t)- \delta)K_t
  $,
  where $\delta>0$ is the depreciation rate
  and $\kappa$ the investment function.
  Goodwin \cite{goodwin1967} originally invokes Say's law, i.e., 
  $
    \kappa~:~  x\in \R_+ \mapsto (1-x)/\nu
  $.
  \item[(III)] Assume a reserve army effect for 
  wage negotiation of the form 
  $dw_t = \Phi(y_t) w_t dt$ where
  $w_t:=a_t x_t$ represents the real wage of the total working population,
  and $\Phi$ is the Phillips curve.
\end{enumerate}
Defining $\gamma:=\alpha+\beta+\delta$
allows to retrieve \eqref{eq:goodwin_model}
for $(x_t, y_t):=(w_t/a_t, K_t/(\nu a_t N_t))$.
The class-struggle model \eqref{eq:goodwin_model}
has been extensively studied because
it allows to generate
endogenous real business cycles
affecting the production level $P_t$, 
e.g. \cite{flashel1984, glombowski1986,grasselli2012, keen1995, vellupillai1979, veneziani2006}.
On this matter,
Goodwin himself conceded that
the model is 
``starkly schematized and hence
quite unrealistic'' \cite{goodwin1967}.
It hardly connects with 
irregular observed trajectories,
see \cite{harvie2000, mohun2006}.

The objective of this paper is thus to
study the following perturbed version of \eqref{eq:goodwin_model}
by a standard Brownian motion $(W_t)_{t\ge 0}$
on a stochastic basis $(\Omega, \Fc, \P)$:

\be\label{eq:stochastic_goodwin_1}
\left\{
\begin{array}{ll}
	d x_t &=x_t \l((\Phi(y_t)-\alpha + \sigma^2(y_t))dt +  \sigma(y_t)dW_t\r) \\
	d y_t &= y_t \l((\kappa(x_t) - \gamma + \sigma^2(y_t)) dt + \sigma(y_t)dW_t \r)
\end{array}
\right. \;,
\ee
where $\sigma$ is a positive function of $y$ bounded by $\sigma_0>0$,
and the filtration $\Fc_t$ is generated by paths of $W$.
The form of $\sigma$ is discussed in Remark \ref{rem:conditions} after.
A stronger condition, Assumption \ref{ass:growth conditions},
is assumed later on the behavior of $\sigma$ to ensure that
solutions of \eqref{eq:stochastic_goodwin_1} remain in $D$.
The example of Section \ref{sec:example}
will also illustrate how such condition can hold.
We modify the economic development (I), (II) and (III) by 
introducing the perturbation on
one assumption, namely
we assume that for $t\ge 0$,

\be\label{eq:stochastic_productivity}
da_t := a_t  d\alpha_t  = a_t \l( \alpha dt - \sigma(y_t)dW_t \r) \;, \quad a_0\ge 0\; ,
\ee
instead of $da_t = a_t \alpha dt$.
Using It\^{o} formula with \eqref{eq:stochastic_productivity} 
in the previous reasoning retrieves \eqref{eq:stochastic_goodwin_1}.
Productivity
is one of the few exogenous parameters of the model,
and one of those that were significantly invoked
as influencial over business cycles, e.g.
\cite{evans1992, hansen1985}.
Without arguing for the pertinence
of that particular assumption,
we simply suggest here that a standard continuous perturbation
in this crucial parameter seems a good starting point.

To our knowledge, 
this is the first attempt to consider random noise in
Goodwin interpretation of the famous prey-predator model.
To stay in the spirit of the economic application,
the present paper studies 
the cyclical behavior of the deterministic
system \eqref{eq:goodwin_model} and 
the stochastic version \eqref{eq:stochastic_goodwin_1}.
Namely, our contribution are as follows, developed
in the present order:
\begin{itemize}
 \item In Section \ref{sec:goodwin}, 
 we fully characterize solutions
 of \eqref{eq:goodwin_model} and the period of 
 their orbits. 
 This generalizes standard results on 
 Lotka-Volterra systems to bounded domains of existence.
 \item In Section \ref{sec:stochastic},
 we provide existence conditions for regular solutions
 of \eqref{eq:stochastic_goodwin_1}.
 We use the entropy of \eqref{eq:goodwin_model}
 to estimate the deviation induced by \eqref{eq:stochastic_productivity}.
 We provide a definition of stochastic orbits for \eqref{eq:stochastic_goodwin_1}.
 The proof that solutions of \eqref{eq:stochastic_goodwin_1}
  draw stochastic orbits in finite time around a unique point
  is given in Section \ref{sec:proof}.
\end{itemize}

Our contribution has to be put in contrast with
numerous studies of random perturbations of
the Lotka-Volterra system.
Apart from the obviously different origin of perturbations
in the model, 
attention was mainly given to systems like \eqref{eq:stochastic_goodwin_1}
for its asymptotic behavior
(e.g., \cite{khasminskii2001, mao2003, nguyen2006}),
regularity, persistence and extinction of species
(e.g., \cite{cai2004, mao2002, nguyen2006, reichenbach2006}),
and the addition of regimes, jumps or delay
(e.g., \cite{bahar2004, liu2014,zhu2009b}).
Here,
we attempt to provide a relevant description of trajectories
$(x_t, y_t)$ and indirectly $P_t$, namely
a cyclical behavior.
This is done using stochastic Lyapunov
techniques for recurrent domains as described in \cite{khasminskii, thygesen1997}.
By conveniently dividing the domain $D$,
we obtain that
almost every trajectory ``cycles'' around a point
in finite time.
The $L_1$-boundedness is out of reach with our method,
but numerical simulations are presented in Section \ref{sec:example},
not only to provide expectation of cycles, but also
allow to conjecture a limit cycle phenomenon
for the expectation of $(x_t, y_t)$.
It is somewhat unclear how 
Assumption \ref{ass:non-linear goodwin} and late Assumption \ref{ass:growth conditions} 
on $\Phi$, $\kappa$ and $\sigma$, 
are relevant in these results.
We show below thatthey are sufficient to obtain existence of 
regular solutions to \eqref{eq:stochastic_goodwin_1}.
This actually emphasizes the role played by
the entropy of the deterministic system
in the well-posedness of the stochastic system
and as a natural measure for perturbation.

\section{Deterministic orbits}
\label{sec:goodwin}

According to Assumption \ref{ass:non-linear goodwin},
there exists only one 
non-hyperbolic
equilibrium point to \eqref{eq:goodwin_model} in $D$
given by 
$(\wt{x}, \wt{y}):= (\kappa^{-1}(\gamma), \Phi^{-1}(\alpha))$.
On the boundary of $D$, there exists
also an additional equilibrium $(0,0)$
which is eluded along the paper.

\begin{Definition}
 \label{def:Lyapunov}
 Let $V_1, V_2$ and $V$ be three functions defined by $V~:~(x,y)\in \R^*_+\x (0,1) \mapsto V_1(x)+V_2(y)$ and
 
 $$
 V_1~:~x\in \R^*_+ \mapsto  \int_{\wt x}^{x}\frac{\kappa(\wt{x}) -  \kappa(s)}{s}ds \;, \quad
 V_2~:~y\in (0,1) \mapsto \int_{\wt y}^{y} \frac{\Phi(s) - \Phi(\wt{y})}{s}ds \;.
 $$
\end{Definition}

\begin{Lemma}
 \label{lem:goodwin system}
 Let $(x_0, y_0)\in D$.
 Let Assumption \ref{ass:non-linear goodwin} hold. Then 
 a solution $(x_t, y_t)$ to \eqref{eq:goodwin_model}
 starting at $(x_0, y_0)$ at $t=0$
 describes closed orbits given by
 the set of points $\{(x,y)\in D ~:~ V(x,y)=V(x_0, y_0)\}$,
 and $(x_t, y_t)\in D$ for all $t\ge0$.
\end{Lemma}

\begin{proof}
 \label{proof:goodwin system}
 It is well-known \cite{grasselli2012} that
 $V$ is a Lyapunov function and a constant of motion for system \eqref{eq:goodwin_model}:
 $V_1$ and $V_2$ take non-negative values with 
 $V_1(\wt x)=V_2(\wt y)=0$, and 
 $dV/dt(x,y) = 0$.
 Additionally, under Assumption \ref{ass:non-linear goodwin}.(i)-(ii),
 
 \begin{equation}
  \label{eq:limits of V boundary}
  \lim_{x\uparrow +\infty}V_1(x)=\lim_{x\downarrow 0^+}V_1 (x) = \lim_{y\uparrow 1^-} V_2(y) = \lim_{y\downarrow 0^+} V_2(y)=+\infty\;,
 \end{equation}
 so that for any $(x_0, y_0)\in D$, $V(x_0, y_0)<+\infty$ and the solution stays in $D$.

\end{proof}

The value of $V$ characterizing an orbit,
it is in bijection with its period.
The following generalizes \cite{hsu1983}.

\begin{Theorem}\label{th:periods_goodwin}
Let $(x_t,y_t)_{t\ge0}$ be
a solution to 
\eqref{eq:goodwin_model}
satisfying Assumption \ref{ass:non-linear goodwin},
with 
$(x_0, x_0)\in D\backslash \{(\wt x, \wt y)\}$.
Let 
$V_0:=V(x_0, y_0)$,
and $\underline{x}<\bar{x}$
the two solutions to equation $V_1(x)=V_0$.
Define three functions $F_1, F_2, G$ by

$$
F_1~:~u\in \R\mapsto  V_2(\Phi^{-1}(u^+ + \alpha))\;, 
\quad F_2~:~u\in \R\mapsto V_2(\Phi^{-1}(-u^- + \alpha))\; ,
\quad G~:~z\in \R\mapsto  V_0 - V_1(e^z)\; .
$$
Then 
$(x_T, y_T) = (x_0, y_0)$
for $T$ defined by

\be
\label{eq:period_goodwin}
T(V_0):=\int_{log(\underline{x})}^{\log(\bar{x})}
\frac{1}{F_1^{-1}( G(z))} - \frac{1}{F_2^{-1}( G(z))}dz \; .
\ee
\end{Theorem}

\begin{proof}
Let $(x_0,y_0)\in D\backslash\{(\wt x, \wt y)\}$, 
$V_0 = V(x_0, y_0)>0$
and $(x_t, y_t)$ a solution to \eqref{eq:goodwin_model}
starting at $(x_0,y_0)$.
According to Lemma \ref{lem:goodwin system},
$V_1(x_t)=V_0$ implies $V_2(y_t)=0$.
Then
$\{x\in \R_+~:~x_t=x \text{ for some }t\ge 0\}=[\underline{x}, \bar{x}]=:I$.
Homogeneity of \eqref{eq:goodwin_model}
allows to set $(x_0, y_0)=(\underline{x}, \wt y)$ without loss of generality.
Let $T_1:=\inf\{t\ge 0~:~x_t = \bar{x}\}$.
For $t\in [0,T_1]$,
$(x_t, y_t) \in 
[\underline{x}, \bar{x}]\x [\wt{y}, \bar{y}]$,
with $\bar{y}$ such that
$V_2(\bar{y})=V_0$.
Let $z_t:=\log(x_t)$ for $t\ge 0$.
Then \eqref{eq:goodwin_model} rewrites
$dz = (\Phi(y) - \alpha)dt$, 
$y = \Phi^{-1}(dz/dt+\alpha)$
and we get

$$
\frac{d^2 z}{dt^2} = \Phi'(y)\frac{dy}{dt}
= \Phi'\l(\Phi^{-1}\l(\frac{dz}{dt}+\alpha\r)\r) \Phi^{-1}\l(\frac{dz}{dt}+\alpha\r)
[\kappa(e^z) - \gamma] \;.
$$
Let
$u:=\Phi(y) - \alpha$
and define 
$\Psi := \Phi'\circ \Phi^{-1} \x \Phi^{-1}$,
to rewrite again

\be\label{eq:second_order_system}
\left\{
\begin{array}{ll}
	dz/dt &= u \\
	du/dt &= \Psi\l(u+\alpha\r) [\kappa(e^z) - \gamma] 
\end{array}
\right. \;.
\ee
Since
$z_t\in [\underline{z}, \bar{z}]:=[\log(\underline{x}), \log(\bar{x})]$
and 
$u_t\in [0, \Phi(\bar{y})-\alpha]$
for $t\in [0,T_1]$,
separation of variables in \eqref{eq:second_order_system}
provides two quantities $F$ and $G$:

\be\label{eq:period_function_1}
F(u):=\int_0^u \frac{s}{\Psi\l(s+\alpha\r)}ds  = 
\int_{\underline{z}}^{z}[\kappa(e^s) - \gamma] ds =: G(z) \; .
\ee
The function $F$ verifies
$F(0)=G(\bar{z})=0$,
is increasing on 
$[0, \Phi(\bar{y})-\alpha]$
and decreasing on 
$[\Phi(\underline{y})-\alpha, 0]$
with $\underline{y}<\wt y$ so that
$V_2(\underline{y})=0$.
Coming back to $y = \Phi^{-1}(u+\alpha)$
we get

$$
F(u)=\int_0^{u}\frac{s}{\Phi' (\Phi^{-1}(s+\alpha))\Phi^{-1}(s+\alpha)}ds
=\int_{\wt y}^{\Phi^{-1}(u+\alpha)} \frac{\Phi(s)-\Phi(\wt y)}{s}ds
=V_2(\Phi^{-1}(u+\alpha)) \; ,
$$
implying that
$F(u)\in [0,V(\underline{x}, \wt y)]$
for $u\in [\Phi(\underline{y})+\alpha, \Phi(\bar{y})+\alpha]$.
We can write $F = F_1 + F_2$ where
$F_1$ and $F_2$ are the two restrictions of $F$
on $\R_+$ and $\R_-$ respectively.
%
Notice that if $t\in[0,T_1]$,
then
$u_t := \Phi(y_t)+ \alpha\in [0,\Phi(\bar{y}+\alpha)]$.
Thus,
$F_1(u_t)$ is a strictly increasing function of $t$
taking its values in 
$[0,V(\underline{x}, \wt y)]$.
Getting back to
$x=e^z$ for $G$,
we have
for $z\ge \underline{z}$

$$
G(z):=\int_{\underline{x}}^{\wt x \wedge e^z} \frac{\kappa(s)-\gamma}{s}ds + 
\int_{\wt x \wedge e^z}^{e^z}\frac{\kappa(s)-\gamma}{s}ds
=V_0-V_1(e^z)
$$
Since 
$\text{sign}(\kappa(x)-\gamma)
=\text{sign}(\wt x - x)$
we have
$
\max_{z\in[\underline{z}, \bar{z}]}G(z) = G(\log(\wt x))
=V_1(\underline{x}) = 
V(\underline{x}, \wt y)$, 
while minimums are given by
$G(\bar{z}) = G(\underline{z}) = 0$.
This sums up with 
$G([\underline{z}, \bar{z}])\subset [0,V_0]$,
so we can write on this interval
$
F_1^{-1}( G(z)) = u = dz/dt
$
which finally gives

$$
T_1 = \int_{\underline{z}}^{\bar{z}} \frac{dz}{F_1^{-1}(G(z))} \;.
$$
We apply 
the same method for 
the other half orbit,
taking 
$(x_0, y_0) = (\bar{x}, \wt y)$
and
$T_2:=\inf\{t \ge 0 ~:~x_t = \underline{x}\}$,
to reach 
the other half of 
expression \eqref{eq:period_goodwin}, 
i.e., $T(V_0)=T_1+T_2$.
\end{proof}

\begin{Remark}
\label{rem:period}
A first order approximation of
\eqref{eq:goodwin_model} 
at $(\wt x, \wt y)$ provides 
a linear homogeneous system,
which solution is trivially given
by a linear combination of 
sines and cosines of 
$(-\wt x\Phi'(\wt y)\wt y\kappa'(\wt x)t)$.
It follows that

$$
\lim\limits_{V_0\rightarrow 0} T(V_0)= \frac{2\pi}{\sqrt{-\wt x \Phi'(\wt y)\wt y\kappa'(\wt x)}}>0\;.
$$
\end{Remark}

\section{Stochastic Goodwin model}
\label{sec:stochastic}


We study a specific
case of \eqref{eq:stochastic_goodwin_1}
where the deterministic part 
cancels at a unique point in $D$
defined by
$(\we x, \we y): = (\kappa^{-1}(\gamma - \sigma^2(\we y)), \we y)$
where $\we y$ comes from the following.

\begin{Assumption}\label{ass:unique_root}
There is  
a unique $\we y\in (0,1)$ such that
$\Phi(\we y) - \alpha + \sigma^2(\we y) = 0$
\end{Assumption}

For a stochastic differential
equation to have a unique global
solution for any given
initial value,
functions $\Phi$
and $\kappa$ are generally
required to
satisfy linear growth
and local Lipschitz conditions,
see \cite{khasminskii}.
We can however consider the following Theorem of Khasminskii \cite{khasminskii},
which is a reformulation of Theorem 3.4, Theorem 3.5 and Corollary 3.1 of \cite{khasminskii}
to our context.

\begin{Theorem}\label{th:existence}
Consider the following stochastic differential equation for 
$z$ taking values in $\R^2_+$:

\begin{equation}
\label{eq:sde}
dz_t = \mu(z_t)dt+\sigma(z_t)dW_t \;.
\end{equation}
Let $(D_n)_{n\ge 1}$ be an increasing sequence of open sets, 
and $(K_n)_{n\ge 1}$ a sequence of constants, verifying
\begin{enumerate}
 \item[(a)] $\bar{D}_n\subset D$ for all $n\ge 1$,
 \item[(b)] $\bigcup_n D_n=D$.
 \item[(c)] For any $n\ge 1$, functions $\mu$ and $\sigma$ are Lipschitz
 on $D_n$ and verify $|\mu(z)|+|\sigma(z)|\le K_n (1+ |z|)$ for any $z\in D_n$.
\end{enumerate}
Let $\vp\in \Cc^{1,2,2}(\R_+\x D)$
and $(K,k)\in \R_+^2$
be such that, denoting $\Lc^z$ the generator associated with \eqref{eq:sde},
\begin{enumerate}
 \item[(d)] $\Lc^z \vp(t, z_t) \le K \vp(t, z_t) + k$ on the set $\R_+\x D$,
 \item[(e)] $\lim_n \inf_{D\backslash D_n} \vp(t,z) = +\infty$ for any $t\ge 0$.
\end{enumerate}
Then for any $z\in D$, there exists
a regular adapted solution to \eqref{eq:sde},
unique up to null sets,
with the Markov property and
verifying $z_t \in D$ for all $t\ge 0$ almost surely.
\end{Theorem}

To satisfy conditions (a) to (e),
we study \eqref{eq:stochastic_goodwin_1} under
the additional sufficient growth conditions.

\begin{Assumption}
 \label{ass:growth conditions}
 There exist two positive constants $K,k$ such that
 \begin{enumerate}
 \item[(i)] $\sigma^2(y)\Phi'(y) \le K V_2 (y)+ k$ for all $y\in (0,1)$,
 \item[(ii)] $-x\kappa'(x)-\kappa(x) \le K V_1(x)+ k $ for all $x\in \R_+^*$.
\end{enumerate}
\end{Assumption}

\begin{Remark}
\label{rem:conditions}
Assumption \ref{ass:growth conditions}.(i) involves both $\Phi$ and $\sigma$ to 
ensure that $y_t\in (0,1)$ for all $t\ge 0$ almost surely.
Assumption \ref{ass:growth conditions}.(ii) holds for polynomial growth of $\kappa$,
suiting the classical conditions of existence on $\R_+$ for $x_t$.
The dependence of $\sigma$ could be generalized to $x$
in full generality, implying a stronger condition than (ii).
We refrain from doing this easy extension, emphasizing the
unavoidable dependence in $y$.
\end{Remark}

For $\varphi\in \Cc^{1,2,2}(\R_+\x D)$,
we recall the diffusion operator
associated with \eqref{eq:stochastic_goodwin_1} by

\begin{equation}
\begin{array}{lll}
\label{eq:generator}
\Lc \vp(t,x,y) &:=& \left[\frac{\partial \varphi}{\partial t}
+ \frac{\partial \varphi}{\partial x}x(\Phi(y)-\alpha + \sigma^2(y) )
+ \frac{\partial \varphi}{\partial y}y(\kappa(x)-\gamma + \sigma^2(y))\right. \\
& & \left.
+\frac{\sigma^2(y)}{2} \left(\frac{\partial^2 \varphi}{\partial x^2}x^2+
 \frac{\partial^2 \varphi}{\partial y^2}y^2 + 2\frac{\partial^2 \varphi}{\partial x \partial y}xy\right)\right] (t,x,y).
 \end{array}
\end{equation}

\begin{Theorem}
\label{prop:existence}
Let $(x_0,y_0)\in D$.
Let Assumptions \ref{ass:non-linear goodwin}, 
\ref{ass:unique_root} and \ref{ass:growth conditions} hold.
Then there exists a solution
$(x_t, y_t)_{t\ge0}$ to \eqref{eq:stochastic_goodwin_1}
staying in $D$ almost surely.
\end{Theorem}

\begin{proof}
Let us show that conditions (a) to (e) of
Theorem \ref{th:existence} are fulfilled.
Consider the sequence of sets $(D_n)_{n\ge 1}$
defined by
$D_n = (1/(n+1), n) \x (1/(n+1), 1-1/(n+1))$.
For any $n\ge 1$, $D_n$ is open and $D_n\subset D_{n+1}$.
(a) and (b) are satisfied with the limit $D=\R_+^*\x (0,1)$.
According to Assumption \ref{ass:non-linear goodwin},
one can always find $K_n$ big enough such that
$\max \{|\Phi(y)-\alpha|; |\kappa(x)-\gamma|\}\le K_n$
for any $(x,y)\in D_n$, and ensures the local Lipschitz condition (c).

Now consider $V$ of Definition \ref{def:Lyapunov}
which is $\Cc^{1,2,2}$ on $D$.
Applying \eqref{eq:generator},

\begin{align*}
	\Lc  V(x,y)
	 = & \l[ \kappa(\wt x) - \kappa(x) \r](\Phi(y)-\alpha+\sigma^2(y)) + 
	 \l[\Phi(y) - \Phi(\wt y) \r](\kappa(x)-\gamma+\sigma^2 (y)) \\
	& + \left(\l[ \kappa(x) - \kappa(\wt x)-x\kappa'(x) \r] +
	\l[ \Phi(\wt y) - \Phi(y)  + y \Phi'(y)\r] 	\right)\sigma^2(y)/2  \;.
\end{align*}
Since $\alpha=\Phi(\wt y)$
and $\gamma = \kappa (\wt x)$,

\begin{equation}
\label{eq:dynkin to V}
\Lc V(x,y) = \left(\l[ \kappa(\wt x) - \kappa(x)-x\kappa'(x) \r] +
	\l[ \Phi(y) - \Phi(\wt y)  + y \Phi'(y)\r] 	\right)\sigma^2(y)/2  \;.
\end{equation}
Assumption \ref{ass:growth conditions}
implies
$\Lc V(x,y)\le \max(\sigma_0^2/2; 2K) V(x,y)+ 2k$
for two positive constants $K, k$,
checking condition (d).
From 
Definition \ref{def:Lyapunov},
$$
\inf_{x\in [0,+\infty)} V(x,y) = V_2(y) + \inf_{x\in [0,+\infty)} V_1 (x)= V_2 (y)
$$
which implies that
$\inf_{D\backslash D_n}V(x,y)\ge \inf\{ V_2(y) ~:~ \max\{y, 1-y\}\le 1/n\}$,
the latter going to infinity
with $n$, recall \eqref{eq:limits of V boundary}.
Similarly,
$\inf_{y\in (0,1)} V(x,y)$
goes to infinity as $x$ goes to $0$
or $+\infty$.
Condition (e) 
is then satisfied,
which  allows to apply
Theorem \ref{th:existence}.

\end{proof}

\begin{Remark}
 \label{rem:basic properties}
 Notice that $\we y<\wt y$ and thus $\we x> \wt x$.
Following Assumption \ref{ass:non-linear goodwin},
\eqref{eq:dynkin to V} at $(\we x, \we y)$ provides
$\Lc V(\we x, \we y) >0$.
It is straightforward that \eqref{eq:stochastic_goodwin_1}
has no equilibrium point in $D$,
nor on its boundary $\{0\}\x(0,1)\cup \R_+^*\x\{0,1\}$.
If the point $(0,0)$ cancels \eqref{eq:stochastic_goodwin_1},
we highlight that
$\Lc V(0,0)<0$ and by continuity,
it holds on a small region $[0,\eps)^2$. 
Recalling \eqref{eq:limits of V boundary} implies that
$(x_t, y_t)$ will diverge from $(0,0)$ almost surely
if $(x_0, y_0)\in D$.
\end{Remark}

A solution to \eqref{eq:stochastic_goodwin_1}
can be pictured
as a trajectory continuously
jumping from an orbit of \eqref{eq:goodwin_model}
to another.
Along this idea, $V$ provides an estimate
on trajectories,
and can be related via Theorem \ref{th:periods_goodwin}
to the period $T$.

\begin{Theorem}
\label{prop:estimate}
Let $(x_0, y_0)\in D$, 
$V_0:=V(x_0, y_0)$, 
and $\l(x_t, y_t\r)_{t\ge0}$
be a regular solution to 
\eqref{eq:stochastic_goodwin_1}.
We first introduce a constant $0\le \rho\le V_0$,
the set $D(V_0,\rho):=\{ (x,y)\in D~:~ |V(x,y)-V_0|\le\rho \}\subset D$
and the stopping time
$\tau_{\rho}:=\inf \{ t>0 ~:~ (x_t, y_t)\notin D(V_0, \rho) \}$.
We then
introduce two finite constants

$$R(V_0, \rho) := \max\limits_{D(V_0,\rho)} \l\{
	\sigma^2(y)\l( \kappa(\wt x) - \kappa(x)-x\kappa'(x)
	+y \Phi'(y) + \Phi(y) - \Phi(\wt y)
	\r)\r\}$$
and 

$$I(V_0, \rho) :=  \max\limits_{D(V_0,\rho)}  \l\{
 	\sigma^2(y)\l( \kappa(\wt x) - \kappa(x)
 	+ \Phi(y) - \Phi(\wt y) \r)^2 \r\}\; .$$
Then for all $\mu>0$

\begin{equation}
\label{eq:estimate}
	\Pro{\tau_\rho > \Theta(\rho,\mu)} 
	\ge \l(1-\frac{I(V_0,\rho)}{\mu^2}\r) \quad \text{ for }\quad 
	\Theta(\rho, \mu) := 
	\frac{2\left(\mu^2 + \mu\sqrt{\mu^2+2\rho R(V_0,\rho)}+\rho R(V_0, \rho)\right)}
	{(R(V_0, \rho)\sigma)^2}\; .
\end{equation}
\end{Theorem}

\begin{proof}
Fix $\mu>0$.
Now we define the $\Fc_{\tau_\rho}$-measurable set
$
	A_{\mu} = \l\{ \omega\in\Omega: \sup_{0< t\le \tau_\rho} \big| M_t(\omega) \big| \le \mu \r\}
$
where $(M_t)_{t\ge 0}$ is a martingale defined by
$M_t=0$ for $t=0$ and for $t>0$ by

$$
	M_t = \frac{1}{\sqrt{t}} \int_0^t  
	\sigma(y_s)\l( \kappa(\wt x) - \kappa(x_s)
	+ \Phi(y_s) - \Phi(\wt y) \r)dW(s)\;.
$$
The process
$M$ is not right-continuous
at $t=0$
but still verifies
$
\Esp{M_t^2}\le I(V_0, \rho)
$
for all $0<t \le \tau_\rho$.
The property holds 
by replacing $M_t$
by its
\cadlag representation.
Doob's martingale inequality
can then be applied:
$\Pro{A_\mu}\ge 1-I(V_0,\rho)/\mu^2$.
At last, 
using It\^{o}'s formula,
we have from \eqref{eq:dynkin to V}:

\begin{eqnarray*}
		|V(x_t, y_t)-V_0| &\le & \demi \int_0^t \sigma^2(y_s)\l| \kappa(\wt x) - \kappa(x_s)-x_s\kappa'(x_s)
		+y_s \Phi'(y_s) + \Phi(y_s) - \Phi(\wt y) \r|ds\\
		& & + \bigg| \int_0^t \sigma(y_s)\l( \kappa(\wt x) - \kappa(x_s)
		+ \Phi(y_s) - \Phi(\wt y) \r)dW_s\bigg|
\end{eqnarray*}
so that on $\{(t,\omega)\in \R_+\x \Omega ~:~ (t,\omega)\in A_{\mu}\x [0,\tau_\rho(\omega)] \}$
$
|V(x_t, y_t)-V_0| \le\demi R(V_0,\rho) t + \mu \sqrt{t}=:S(t)
$ almost surely.
Also, 
$|e(t,\omega)|\le \rho$ on that set.
Put in another way,
$\tau_\rho > S^{-1}(\rho)=:\Theta(\rho)$ on $A_\mu$.
According to Bayes rule,

$$
	\Pro{\tau_\rho > \Theta(\rho)} \ge 
	\Pro{\tau_\rho > \Theta(\rho) \big| A_{\mu}} \Pro{A_{\mu}} \ge \Pro{A_{\mu}} 
	\ge	\l(1-\frac{I(V_0,\rho)}{\mu^2}\r) 
$$
\end{proof}

We now introduce the main result of the paper.
We provide the following tailor-made
definition for the cycling behavior of \eqref{eq:stochastic_goodwin_1}.

\begin{Definition}
 \label{def:stochastic orbit}
 Let $(x^* ,y^*)\in E\subseteq \R^2$ and $(x_0, y_0)\in E\backslash \{(x^* ,y^*)\}$.
 Let $(x_t, y_t)$ be a stochastic process
 starting at $(x_0, y_0)$ staying in $E$ almost surely. 
 We then introduce $(\rho_t)_{t\ge 0}$ the angle
 between $[x_t-x^*, y_t- y^*]^{\top}$
 and $[x_{0}-x^*, y_{0}-y^*]^{\top}$.
 Let $S:=\inf \{t>0~:~  |\rho_t| \ge 2\pi \text{ or } (x_t, y_t)=(x^* ,y^*) \}$ be a stopping time
 (a stochastic period). Then,
 the process $(x_t, y_t)$ is said to orbit stochastically
 around $(x^*, y^*)$ in $E$ if $S<+\infty$ almost surely.
\end{Definition}

\begin{Theorem}
\label{th:finite period}
Let $(x_0, y_0)\in D\backslash\{(\we x, \we y)\}$
and $(x_t, y_t)$ a solution to \eqref{eq:stochastic_goodwin_1}
starting at $(x_0, y_0)$.
Then $(x_t, y_t)$ orbits stochastically around $(\we x, \we y)$ in $D$.
\end{Theorem}

More precisely the system \eqref{eq:stochastic_goodwin_1}
produces clockwise orbits inside $D$.
The angle $\rho_t$ is only defined if $(x_t, y_t)\ne (\we x,\we y)$.
This can be ensured by either proving that
$(x_t, y_t)\ne (\we x,\we y)$ for all $t\ge 0$ almost surely,
or by defining $S$ as in Definition \ref{def:stochastic orbit}.
See also Remark \ref{rem:out of point}
The proof of Theorem \ref{th:finite period} is 
removed to Section \ref{sec:proof}.

\section{Proof of Theorem \ref{th:finite period}}
\label{sec:proof}


\subsection{Preliminary definitions and results}
\label{sec:preamble}

Recall that the probability space is given by $(\Omega, \Fc, \P)$
with the filtration generated only by $W$ is completed with null sets.
Our proof,
although unwieldy,
allows us to describe
precisely the possible trajectories
of solutions of \eqref{eq:stochastic_goodwin_1}.
it consists in defining
subregions $(R_i)_i$ of the domain $D$,
see Definition \ref{def:regions} below
illustrated by Fig. \ref{fig:rotation},
and prove that
the process exits from them
in finite time
by the appropriate frontier.
According to 
Theorem \ref{prop:existence} any regular solution
of \eqref{eq:stochastic_goodwin_1}
is a Markov process.
We then repeatedly
change the initial
condition of the system,
as equivalent of 
a time translation and use Definition \ref{def:exit times} herefater.
We obtain recurrence properties via Theorem 3.9 in \cite{khasminskii}.
Since it is repeatedly used hereafter, we provide here
a version suited to our context.

\begin{Theorem}
\label{th:reccurent domain}
Let $(x_t, y_t)_{t\ge 0}$
be a regular solution of \eqref{eq:stochastic_goodwin_1} in $D$,
starting at $(x_0, y_0)\in U$,
for some $U\subset D$.
Let $\vp(t,x,y)\in \Cc^{1,2,2}(\R_+\x U)$
verifying $\vp(t,x,y)\ge 0$ for all $(t,x,y)\in U$ and
$
\Lc \vp(s,x,y)\le -\vp(s)
$
where $\vp(s)\ge 0$ and $\lim_t \int_0^t \vp(s)ds=+\infty$.
Then $(x_t, y_t)$ leaves the region $U$ in finite time
almost surely.
\end{Theorem}

\begin{Definition}
 \label{def:theta}
 Let $f$ be defined by 
 $f~:~x\in \R_+ \mapsto f(x):= \Phi^{-1}(\alpha-\gamma+\kappa(x))$
 as a concave decreasing function.
 For a solution $(x_t, y_t)$ to \eqref{eq:stochastic_goodwin_1},
 we define $\theta_t:=y_t/x_t$ the finite variation process verifying
 $
  d\theta_t =  \theta_t\l(\kappa(x) - \gamma + \alpha - \Phi(y)\r)dt
  = \theta_t \l( \Phi(f(x)) - \Phi(y) \r)dt
 $.
 Additionally, let $\we \theta: =  \we y/\we x$.
\end{Definition}

\begin{Definition}
 \label{def:regions}
 We define eight sets $(R_i)_{i=1,\ldots, 8}$
 such that $\bigcap_{i=1}^8 R_i = (\we y, \we x)$
 and $\bigcup_{i=1}^8 R_i=D$, by
 
$$
\left\{
\begin{array}{l}
R_1 := \{ (x, y)\in D~:~y\ge\we y\And \theta_t\le \we \theta \}\\
 R_2 := \{ (x, y)\in D~:~ f(x)\le y\le \we y \}\\
R_3 := \{ (x, y)\in D~:~ y \le f(x)\And x\ge \we x \}\\
R_4 := \{ (x, y)\in D~:~ x\le \we x\And \theta \le \we \theta \}\\
R_5 := \{ (x, y)\in D~:~ y\le\we y\And \theta_t \ge \we \theta \}\\
R_6 := \{ (x, y)\in D~:~ \we y\le y \le f(x) \}\\
R_7 := \{ (x, y)\in D~:~ y\ge f(x)\And x_t\le \we x \}
\\
R_8 := \{ (x, y)\in D~:~ y\ge f(x)\And x_t\le \we x \} .
\end{array}
\right.
$$
\end{Definition}

\begin{Definition}
 \label{def:exit times}
 Let $(x_t, y_t)$ be a solution to \eqref{eq:stochastic_goodwin_1}
 starting at $(x_0,y_0)=(x,y)\in D$. For any $i\in \{1, \ldots, 8\}$, 
 we define the stopping times $\tau_i (x, y):=\inf\{t\ge 0~:~ (x_t, y_t)\in R_i \}$.
\end{Definition}

\begin{figure}[htp]
\includegraphics[width=\textwidth, height=7cm]{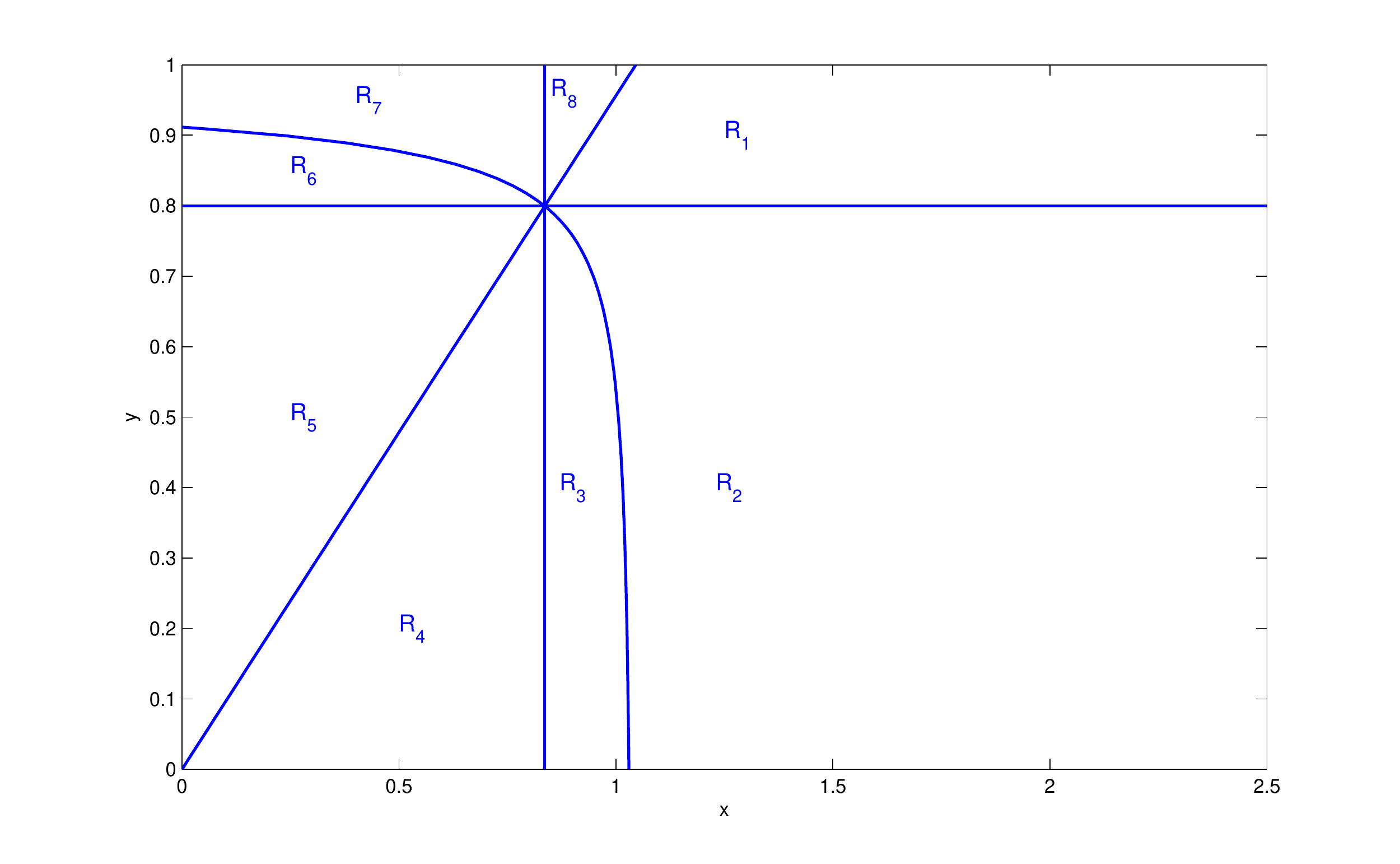}
\caption{Covering 
of $D:=\R_+^*)\x (0,1)$ by $(R_i)_{i=1\ldots 8}$.
Since $f(0)<1$
and $\lim_{y\uparrow1}\Phi(y)=+\infty$,
the graph 
illustrates the general case.}
\label{fig:rotation}
\end{figure}

%
\begin{Remark}
 \label{rem:out of point}
It seems rather clear that
the point $(\we x, \we y)$ is
not reached in finite time with a positive probability.
In the following,
the fact that $\Lc V(x,y)>\eps$ 
for some small $\eps>0$ in a neighborhood
of $(\we x, \we y)$ implies
that $(\we x, \we y)$ is not
a limit to almost every path of a solution to \eqref{eq:stochastic_goodwin_1},
recall Remark \ref{rem:basic properties}.
\end{Remark}

To ease the reading of the proof of Theorem \ref{th:finite period}
which follows from the following
Propositions \ref{prop:R1_R8} to \ref{prop:R8_R1},
we divide it in four quadrants around $(\we x, \we y)$.
We first prove 
that the process cycles,
even in infinite time,
for some particular starting points.

\begin{Proposition}
\label{prop:R1_R8}
If $(x_0, y_0)\in R_1$,
then
$\Pro{\tau_8(x_0, y_0)\le \tau_7(x_0, y_0)}=0$.
If $(x_0, y_0)\in R_5$, then
$\Pro{\tau_4(x_0, y_0)\le \tau_3(x_0, y_0)}=0$.
\end{Proposition}

\begin{proof}
This is a direct consequence
of the absence of Brownian motion in
$\theta$.
Take $(x_0, y_0)\in R_1$.
Then on $[0,\tau_3(x_0, y_0)]$,
the process $\theta$ is
non increasing almost surely,
meaning that
$R_8$ cannot be reached
without first crossing region $R_7$.
The other side is identical.
 \end{proof}

\begin{Remark}
\label{rem:R1_R2_R3_R4}
Proposition \ref{prop:R1_R8} holds
even if $\tau_i=+\infty$,
for any $i$ involved.
It also implies that
if $(x_0, y_0)\in R_1\cup R_5$,
then
$
\tau_i(x_0, y_0)\le \tau_{j}(x_0, y_0)$
almost surely for $j \in \{mod(i+ 1,8)\}
$.
\end{Remark}

\subsection{Eastern quadrant}
\label{sec:west}

We ought to prove that
for $(x_0,y_0)\in R_1$,
the process reaches $R_3$ in finite time almost surely.

\begin{Proposition}
\label{prop:R1_R2}
If $(x_0, y_0)\in R_1$
then $\Pro{\tau_2(x_0, y_0)<+\infty}=1$.
\end{Proposition}

\begin{proof}
Let $\vp~:~y \in [0,1]\mapsto \sqrt{y}$.
Then $\vp(y)\ge \vp (\we y)>0$ for any $y$
such that $(x,y)\in R_1$.
Moreover,

$$
\Lc \vp(y) = \frac{\vp(y)}{2}\l(\kappa(x)-\gamma + \frac{3}{4}\sigma^2(y)\r)
\le -\frac{\sigma^2(\we y)\vp(y) }{8}\le -\frac{\sigma^2(\we y)h(\we y)}{8}<0 \;.
$$
Theorem \ref{th:reccurent domain} stipulates
that $(x_t,y_t)$ leaves $R_1$
in finite time almost surely which is only possible via
$R_2$ according to 
Proposition \ref{prop:R1_R8}.
Reaching the boundary is prevented by Theorem \ref{prop:existence}.
 
\end{proof}

\begin{Proposition}
\label{prop:leave R_2 and R_3}
If $(x_0, y_0)\in R_2\cup R_3$ then
$\Pro{\tau_1(x_0,y_0)\wedge \tau_4(x_0, y_0)<+\infty}=1$.
\end{Proposition}

\begin{proof}
We follow the proof of
Proposition \ref{prop:R1_R2}
with $\vp~:~x\in \R_+ \mapsto \sqrt{x}$.
 
\end{proof}

\begin{Proposition}
\label{prop:R_1 to R_3}
If $(x_0, y_0)\in R_1\cup R_2$ then
$\Pro{\tau_3(x_0, y_0)<+\infty}=1$.
\end{Proposition}

\begin{proof}
{\it Step 1.}
Let $(\upsilon_n)_{n\ge0}$ be a sequence of stopping times
defined by $\upsilon_0=0$ and 

$$\upsilon_n:=\inf \{t\ge\upsilon_{n-1}~:~y_t 
= \we y \text{ or }(x_t,y_t)\in R_3 \}, \quad n\ge 1.$$
By construction if 
$(x_{\upsilon_n}, y_{\upsilon_n})\in R_3$ for some $n \ge 1$,
then $\upsilon_k=\upsilon_n$ for all $k>n$.
Following Propositions \ref{prop:R1_R8}, \ref{prop:R1_R2} and
\ref{prop:leave R_2 and R_3},
$\upsilon_n<+\infty$ for all $n\ge1$ almost surely,
and
$\{\tau_3(x,y)=+\infty\} \subset \bigcap_{n\ge 1}\{ y_{\upsilon_n}=\we y \}$.
We prove in step 2 that
this implies

\begin{equation}
\label{eq:theta to zero}
\lim\limits_{t \rightarrow \infty} \theta_t (\omega)=0, \text{ for }\Pae \ \omega\in \{\tau_3(x,y)=+\infty\} \; .
\end{equation}
Providing that \eqref{eq:theta to zero} holds
we immediately get
$
\Pro{\tau_3(x_0, y_0)=+\infty}\le \Pro{\lim_n x_{\upsilon_n}=+\infty}=0
$.

{\it Step 2.}
If $\omega \in \{\tau_3(x,y)=+\infty\}$,
then
for all $n\ge 1$,
$y_{\upsilon_n}=\we y$
and according to
Proposition \ref{prop:leave R_2 and R_3},
$(x_t,y_t)$ does not converge to the set
$R_2\cap R_3$.
Since $\theta_t$ is a positive decreasing process
for $(x_t, y_t)\in R_1\cup R_2$,
Doob's martingale convergence theorem
implies that
$\theta_t$ converges
pathwise
in $L^{\infty}([0,\we \theta))$.
Assume now that
$\theta_t$ does not converge to $0$ with $t$
on $E\subset\{\tau_3(x_0, y_0)=+\infty\}$.
Then for any $\eps>0$,
and for almost every $\omega\in E$

\begin{equation}
\label{eq:integral explodes}
\lim_t \int_0^t \ind{\kappa(x_s(\omega)) - \gamma + \alpha - \Phi(y_s(\omega))<-\eps} ds = C_{\eps}(\omega)<+\infty \; .
\end{equation}
If the integral \eqref{eq:integral explodes}
explodes to $+\infty$
for some $\eps>0$ on some 
non null subset $F\subset E$,
then for almost every $\omega \in F$,

$$
\Lc \log\theta_t(\omega) = \l(\kappa(x_t(\omega)) - \gamma + \alpha - \Phi(y_t(\omega))\r)
 < -\eps \ind{\kappa(x_t(\omega)) - \gamma + \alpha - \Phi(y_t(\omega))<-\eps}
$$
and
$
\lim_{t\uparrow\infty}\log\theta_t (\omega)=-\infty$
for almost every $\omega\in F$,
implying that $\theta_t$ converges to $0$ on $F$,
a contradiction with $F\subset E$, so that
\eqref{eq:integral explodes}
holds on $E$.
We then consider the random time $t_{\eps, n}$,
being the first time
such that 

\be
\label{eq:limite_borne}
\int_0^{t_{\eps, n}} \ind{\kappa(x_s) - \gamma + \alpha - \Phi(y_s)<-\eps} ds \ge C_{\eps}-\frac{1}{n}\; ,
\ee
and $k_n$ the smallest index such that
$\upsilon_{k_n}\ge t_{\epsilon, n}$.
Note that $t_{\eps,n}$ is not a $\F$-stopping time
and $k_n$ is not $\F$-adapted
since they depend on $C_\eps$ which is $\Fc_{\infty}$-measurable.
\eqref{eq:limite_borne} implies
that there exists a random time 
$s_{n}\in (\upsilon_{k_n}, \upsilon_{k_n}+1/n)$
such that
$-\eps<\kappa(x_{s_{n}}) - \gamma + \alpha - \Phi(y_{s_{n}})<0$,
otherwise we would have 
a contradiction of 
\eqref{eq:integral explodes}
on a subset of $E$:

$$
\int_0^{\upsilon_{k_n}+1/n} \ind{\kappa(x_s) - \gamma + \alpha - \Phi(y_s)<-\eps} ds \ge C_{\eps}\;.
$$
This implies that 
$\lim_n (s_n - \upsilon_{k_n})(\omega)=0$
for almost every $\omega \in E$,
and $(y_t)_{t\ge 0}$
being a continuous process

$$
\lim_n y_{s_n}(\omega)=\we y, \quad \text{ for }\Pae \ \omega\in D\subset \{\tau_3(x_0,y_0)=+\infty\}\;.
$$
This is impossible
for $\eps>0$ small enough since
$\theta_t$ is strictly decreasing and
thus $E$ is a null set.
\eqref{eq:theta to zero} holds.
\end{proof}

\subsection{Southern quadrant}
\label{sec:south}

We show that starting from $R_2\cap R_3$,
$(x_t,y_t)$ reaches $R_5$ in finite time almost surely.

\begin{Proposition}
\label{prop:R2_to_R4}
If $(x_0,y_0)\in R_2\cap R_3$
then $\Pro{\tau_4(x_0,y_0)<+\infty}=1$.
\end{Proposition}

\begin{proof}
{\it Step 1.}
We consider
$\vp_t := \vp(x_t, y_t)$
with 
$\vp~:~(x,y)\in D\backslash\{(x,y)\}\mapsto(y_t - \we y)/(x_t-\we x)$,
and
aim to prove that
the process 
$
F_t:=F(\vp(x_t, y_t))$
with 
$F~:~\vp\in (-\pi/2, \pi/2)\mapsto  \tan\l( \tan^{-1}(\vp) + \tan^{-1}(c) \r)
$
is a supermartingale on $R_1 \cup R_2 \cup R_3$,
for $c \in (0, \we\theta^{-1})$.
Notice that it is bounded in $R_1 \cup R_2 \cup R_3$. 
Applying It\^{o} to $\vp$ first gives

$$
d\vp_t =
\frac{dy_t}{x_t-\we x} -\frac{y_t-\tilde y}{(x_t-\we x)^2}dx_t
+\frac{\sigma^2(y_t)}{(x_t-\we x)^2} \l[\frac{y_t-\we y}{x_t-\we x} x_t^2 -  x_ty_t\r]dt\; .
$$
Then,
noticing that $F_t =(\vp_t+c)(1-\vp_t c)$,
we obtain

$$
d F_t = \frac{1+c^2}{(1-\vp_t c)^2}\l( d\vp_t + \frac{c}{1-\vp_t c}d\left\langle \vp\right\rangle_t \r) \;.
$$
It is clear that
$
-(y - \we y)(\Phi(y)- \alpha + \sigma^2(y))\le 0$
for all $y\in [0,1)$.
Now notice that for $(x, y)\in R_1$, 
we have
$(x-\we x)(\kappa(x)-\gamma+\sigma^2(y))<0$
so that

\begin{align*}
	\frac{(x-\we x)^4}{\sigma^2(y)}\frac{(1-\vp c)^2}{1+c^2}\Lc F &
	\le (y-\we y)(x-\we x)x^2-xy(x-\we x)^2+\frac{1}{\we y /\we x -\vp }(\we yx-y\we x)^2\\
	&=(x-\we x)\l[x^2(y-\we y)-xy(x-\we x)+\we x(\we yx-y\we x)\r]\\
	&=(x-\we x)^2 x \we x\l[\frac{y}{x}-\frac{\we y}{\we x}\r] < 0 \; .
\end{align*}
Now on $R_2\cup R_3$,
$\wt x < \we x$
implies that
$(\kappa(x)-\gamma)<0$,
so that 

\begin{align*}
	\frac{(x-\we x)^4}{\sigma^2(y)}\frac{(1-\vp c)^2}{1+c^2}\Lc F 
	&\le (y-\we y)(x-\we x)x^2-y\we x(x-\we x)^2+\frac{1}{\we y/\we x-\vp}(\we yx-y\we x)^2\\
	&= (x-\we x)\l[x^2(y-\we y)-y\we x(x-\we x)+\we x(\we yx-y\we x)\r]\\
	&= (x-\we x)^2x\l(y-\we y\r)<0 \;.
\end{align*}
Denoting 
$\tau_{1,4}:= \tau_1(x_0,y_0)\wedge \tau_4(x_0,y_0)$,
we conclude that 
$F_{t \wedge \tau_{1,4}}$ 
is a supermartingale 
for $t\ge 0$.
Using optional sampling theorem,
assisted by Proposition \ref{prop:leave R_2 and R_3},
$\tau_{1,4}<+\infty$ almost surely and

$$
F_0 \ge \Esp{F_{\tau_{1,4}}}
= \frac{1}{c}\Pro{\tau_4(x,y)<\tau_1(x,y)}
+ c \Pro{\tau_1(x_0,y_0) < \tau_4(x_0,y_0)}
$$
Since 
$M:=\max\{F(\vp(x,y))~:~ (x, y)\in R_2\cap R_3\}<c$
then

$$
\Pro{\tau_4(x_0,y_0)<\tau_1(x_0,y_0)} \ge \frac{c(c-M)}{c^2 +1}>0 \quad \forall (x_0,y_0)\in R_2\cap R_3\;.
$$
{\it Step 2.}
According to 
Proposition \ref{prop:leave R_2 and R_3},
$\tau_{1,4}<+\infty$
almost surely
for any  $(x_0,y_0)\in R_2\cap R_3$,
and according to
Proposition \ref{prop:R_1 to R_3},
$\tau_3(x_0,y_0)<+\infty \ \Pas$ 
for all $(x_0,y_0)\in R_1$.
Taking $(x_0,y_0)\in R_2\cap R_3$,
we define the sequence
$(\tau_{1,4}^n,\tau_3^n)_{n\ge 0}$
with $\tau_3^0=0$ and 

$$
\l\{
\begin{array}{ll}
\tau_{1,4}^n :=& \inf\{ t\ge \tau_3^{n}~:~(x_t,y_t)\in R_1 \cup R_4 \}\\
\tau_3^{n+1} :=& \inf\{t\ge\tau_{1,4}^n~:~ (x_t,y_t)\in \left(R_2\cap R_3\right)\cup R_4\}
\end{array}
\r. , \ \pourtout n\ge 1\;.
$$
We then have
$\{\tau_4(x_0,y_0)=+\infty \}\subset \cap_{n \ge 1}\{ x_{\tau_{1,4}^n}>\we x \}$
for any $(x_0,y_0)\in R_2 \cap R_3$.
The sequence $(\{ x_{\tau_{1,4}^n}>\we x \})_{n\ge 1}$
is decreasing in the sense of inclusion,
so that

\begin{equation}
\label{eq: for tau_4}
\Pro{\tau_4 (x,y)=+\infty}=\lim_n \Pro{x_{\tau_{1,4}^n}>\we x}\;.
\end{equation}
Using Baye's rule,

$$
\Pro{x_{\tau_{1,4}^n}>\we x}
\le \Prod_{k=1}^n \Pro{x_{\tau_{1,4}^k}>\we x | x_{\tau_{1,4}^{k-1}}>\we x }
\le \Prod_{k=1}^n \Pro{x_{\tau_{1,4}^k}>\we x | x_{\tau_{3}^k}>\we x }\; .
$$
Using step 1 of the present proof and the Markov property of $(x_t,y_t)$,

$$
\Pro{x_{\tau_{1,4}^n}>\we x}\le  \Prod_{k=1}^n \Pro{\tau_1(x_{\tau_{3}^k},y_{\tau_{3}^k})<\tau_4(x_{\tau_{3}^k},y_{\tau_{3}^k})}
\le \Prod_{k=1}^n \left(1- \frac{c(c-M)}{c^2 +1}\right)\; .
$$
Plugging this inequality into \eqref{eq: for tau_4}
concludes the proof.
 \end{proof}

\begin{Remark}
Notice that by choosing $c$ properly
in the above proof,
it is possible to be arbitrarily close
to $R_5$ in finite time.
The device is used later
in Proposition \ref{prop:R7_R8}.
\end{Remark}

\begin{Proposition}
\label{prop:R4_to_R5}
If $(x_0, y_0)\in R_3\cap R_4$ then
$\Pro{\tau_5(x_0,y_0)<+\infty}=1$.
\end{Proposition}

\begin{proof}
{\it Step 1.}
We claim that $\tau_{2,5}:=\tau_2(x_0, y_0)\wedge \tau_5 (x_0, y_0)<+\infty$
almost surely.
Consider the function $\vp~:~(x,y)\in D \mapsto\sqrt{x_{t\wedge \upsilon_0}}$.
The process $\vp_t:=\vp(x_t, y_t)$ is a positive supermartingale on $R_2 \cup R_3 \cup R_4$:

\begin{equation}
\label{eq:dynkin_h}
\Lc \vp(x,y)=\frac{\vp(x,y)}{2}\left( (\Phi(y)-\alpha+\frac{\sigma^2 (y)}{2}\right)
\le -\frac{\sigma^2 (y)\vp(x,y) }{4}\;.
\end{equation}
According to Doob's martingale convergence theorem,
$\vp_t$ converges point-wise with $t$.
Let $\eps>0$ and define
$R_\eps:=\bigcup_{i=2}^4 R_i \cap \{ x\ge \eps \}$.
Then $\vp_t\ge \sqrt{\eps}$ on $R_{\eps}$,
and similarly to Proposition \ref{prop:leave R_2 and R_3},
we use Theorem \ref{th:reccurent domain} 
to assert that $(x_t,y_t)$
leaves $R_\eps$ in finite time almost surely.
This being true for any $\eps>0$,
$\lim _t \vp_t(\omega)=0$
for almost every $\omega\in \{\tau_{2,5}(\omega)=+\infty\}$.
In $R_5$,
this is only possible if $\lim_t y_t(\omega)=0$
also,
implying that $\lim_t (x_t(\omega), y_t(\omega))=(0,0)$ on this set.
This being improbable, 
$\tau_{2,5}<+\infty$ almost surely.

{\it Step 2.}
By denoting $\tau_4^0=0$,
we then define the sequence $(\tau_{2,5}^n, \tau_4^n)_{n\ge0}$
by

$$
\l\{
\begin{array}{ll}
\tau_4^n :=& \inf\{t\ge\tau_{2,5}^{n-1}~:~ x_t=\we x \text{ or }(x_t,y_t) \in R_5\}\\
\tau_{2,5}^n :=& \inf\{ t\ge \tau_4^{n}~:~(x_t,y_t)\in R_2 \cup R_5 \}
\end{array}
\r. , \ \pourtout n\ge 1\;.
$$
If $(x_{\tau_{2,5}^0}, y_{\tau_{2,5}^0})\in R_2$,
then, according to
Proposition \ref{prop:R2_to_R4}, 
the process reaches back
$R_4$ in finite time. 
Using step 1,
we have that
$\Pro{\tau_4^n<+\infty}=\Pro{\tau_{2,5}^n<+\infty}=1$.
By construction and 
Proposition \ref{prop:R2_to_R4}, 
for $n\ge 1$

\begin{equation}
\label{eq:back to R_2}
\{(x_{\tau_{2,5}^n},y_{\tau_{2,5}^n})\in R_2\}
\subset\{x_{\tau_4^n}=\we x\}
= \{(x_{\tau_{2,5}^{n-1}}, y_{\tau_{2,5}^{n-1}})\in R_2\}=\{x_{\tau_{2,5}^{n-1}}>\we x\}.
\end{equation}
Therefore,
$\{\tau_5(x_0, y_0)=+\infty\}=\bigcap_{n\ge 0}\{(x_{\tau_{2,5}^n},y_{\tau_{2,5}^n}) \in R_2\}$
and
the sequence of sets 
$$\left(\{(x_{\tau_{2,5}^n},y_{\tau_{2,5}^n}) \in R_2\}\right)_{n\ge 0}$$
is decreasing
in the sense of inclusion.
Altogether
we get

\begin{equation}
\label{eq: never reach R_5}
\Pro{\tau_5(x_0, y_0)=+\infty} = \lim_n \Pro{ x_{\tau_{2,5}^n}> \we x}\;.
\end{equation}
Now using Bayes formula
and \eqref{eq:back to R_2},
we finally obtain for every $n\ge1$

\begin{equation}
\label{eq:prod de probas}
\Pro{x_{\tau_{2,5}^n}> \we x}\le\Prod_{k=1}^n \Pro{x_{\tau_{2,5}^k}> \we x | x_{\tau_{2,5}^{k-1}}> \we x}
= \Prod_{k=1}^n \Pro{x_{\tau_{2,5}^k}> \we x | x_{\tau_4^{k}}= \we x}
\end{equation}
Putting \eqref{eq: never reach R_5} and \eqref{eq:prod de probas}
together,
$\Pro{\tau_5(x_0, y_0)=+\infty}>0$ implies that

\begin{equation}
\label{eq:limite}
\lim_n \Pro{x_{\tau_{2,5}^n}>\we x | x_{\tau_4^{n}}= \we x}=1 \;.
\end{equation}

{\it Step 3.}
Let $\vp~:~(t,x)\in \R_+^2 \mapsto \sqrt{x}\exp(\frac{1}{8}\sigma^2(\we y)t)$.
According to \eqref{eq:dynkin_h}
the process $\vp_t:=\vp(t, x_t)$ 
is a supermartingale 
on $[\tau_4^{n}, \tau_{2,5}^{n}]$.
Fixing $t>0$ and applying
optional sampling theorem,
we obtain

$$
\Esp{\vp(t\wedge \tau_{2,5}^{n}, x_{t\wedge \tau_{2,5}^{n}})- 
\vp(t\wedge \tau_4^{n}, x_{t\wedge \tau_4^{n}})| x_{t\wedge \tau_4^n}=\we x}\le 0 \;.
$$
Since $\max(\tau_4^{n},\tau_{2,5}^{n})<+\infty$ almost surely,
we apply Fatou's lemma and obtain

\be
\label{eq:supermartingale_2}
\Esp{\exp \left(\frac{1}{8}\sigma^2(\we y)[\tau_{2,5}^n-\tau_4^{n}]\right)
\sqrt{x_{\tau_{2,5}^n}}\left(\ind{x_{\tau_{2,5}^n}<\we x}+
\ind{x_{\tau_{2,5}^n}>\we x}\right) 
\big| x_{\tau_4^{n}}=\we x} 
\le \sqrt{\we x}\; .
\ee
Since $\sqrt{x_{\tau_{2,5}^n}}\ind{x_{\tau_{2,5}^n}<\we x}\ge 0$ and 
$\sqrt{x_{\tau_{2,5}^n}}\ind{x_{\tau_{2,5}^n}>\we x} 
\ge \sqrt{\we x}\ind{x_{\tau_{2,5}^n}>\we x}$ 
for all $n\ge 1$,
\eqref{eq:supermartingale_2} implies

$$
\Esp{\exp\left(\frac{1}{8}\sigma^2(\we y)[\tau_{2,5}^n-\tau_4^{n}]\right)
\ind{x_{\tau_{2,5}^n}>\we x}
\big| x_{\tau_4^{n}}=\we x}
\le 1 \; ,
$$
leading to

\be
\label{eq:probable ineq}
\Esp{\left(\exp\left(\frac{1}{8}\sigma^2(\we y)[\tau_{2,5}^n-\tau_4^{n}]\right)-1\right)
\ind{x(\tau_{2,5}^n)>\we x}
\big| x_{\tau_4^{n}}=\we x}
\le 1 - \Pro{x_{\tau_{2,5}^n}>\we x | x_{\tau_4^{n}}=\we x} \;.
\ee
If $x_{\tau_4^n}=\we x$ then
$y_{\tau_4^{n}}<\we y$ and
by continuity
$\{\tau_{2,5}^n>\tau_4^{n}\}\supset \{x_{\tau_4^n}=\we x\}$,
implying

\be
\label{eq:exp moins un}
\exp\left(\frac{1}{8}\sigma^2(\we y)[\tau_{2,5}^n(\omega)-\tau_4^{n}(\omega)]\right)>1,\  
\text{ for }\Pae \ \omega \in  \{x_{\tau_4^n}=\we x\}\;.
\ee

Let's assume that 
$\Pro{\tau_5 (x_0, y_0)=+\infty}>0$,
so that
\eqref{eq:limite}
holds.
According to
\eqref{eq:probable ineq},
we get

$$
0\le \Esp{\left(\exp\left(\frac{1}{8}\sigma^2(\we y)[\tau_{2,5}^n-\tau_4^{n}]\right)-1\right)
\ind{x_{\tau_{2,5}^n}>\we x}
\big| x_{\tau_4^{n}}=\we x}
\rightarrow 0 \text{ as } n \rightarrow \infty  \; .
$$
Markov inequality then leads to the following
convergence for any $\eps>0$:

$$
\lim_n \Pro{\left(\exp\left(\frac{1}{8}\sigma^2(\we y)[\tau_{2,5}^n-\tau_4^{n}]\right)-1\right)
\ind{x_{\tau_{2,5}^n}>\we x}>\eps
\big| x_{\tau_4^{n}}=\we x}
=0 \; .
$$
Now Bayes rules with \eqref{eq:exp moins un}
provides

\begin{align*} 
\label{eq:bayes proba}
& \Pro{\left(\exp\left(\frac{1}{8}\sigma^2(\we y)[\tau_{2,5}^n-\tau_4^{n}]\right)-1\right)>\eps
\big| x_{\tau_{2,5}^n}>\we x , \; x_{\tau_4^{n}}=\we x}
\Pro{x_{\tau_{2,5}^n}>\we x | x_{\tau_4^{n}}=\we x}
\\ & \hspace{3cm} =
\Pro{(\tau_{2,5}^n-\tau_4^{n})>8\ln(1+\eps)/\sigma^2(\we y)
\big| x_{\tau_{2,5}^n}>\we x , \; x_{\tau_4^{n}}=\we x}
\Pro{x_{\tau_{2,5}^n}>\we x | x_{\tau_4^{n}}=\we x}
\end{align*}
which leads for any $\eps>0$ to
$\Pro{(\tau_{2,5}^n-\tau_4^{n})>\eps | x_{\tau_{2,5}^n}>\we x,\; x_{\tau_4^{n}}=\we x }
\rightarrow 0  \text{ as } n \rightarrow \infty$.
From step 2,
$\{\tau_5(x_0, y_0)=+\infty\} = \bigcap_{n\ge 0}\left(\{ x_{\tau_4^n}=\we x \}\cap\{ x_{\tau_{2,5}^n}>\we x \}\right)$.
Therefore on this set,
the continuous mapping theorem asserts that
$(x,y)$ at consecutive stopping times
converge in probability.
By continuity,
this implies
$\lim_n y_{\tau_4^n}(\omega)= \we y$
and $\lim_n x_{\tau_{2,5}^n}(\omega) = \we x$
for $\Pae \ \omega \in \{\tau_5(x_0, y_0)=+\infty \}$.
By the Markov property of $(x_t,y_t)$,
$
\lim\limits_{t\rightarrow\infty}(x_t, y_t)(\omega)=(\we x, \we y)$
for almost every $\omega\in \{\tau_5(x_0, y_0)=+\infty \} \;.
$
We conclude that
$\Pro{\tau_5(x_0, y_0)=+\infty}=0$.
 \end{proof}

\subsection{Western Quadrant}

\begin{Proposition}
\label{prop:R5_to_R7}
If $(x_0,y_0)\in R_5\cup R_6$ then
$\Pro{\tau_7(x_0,y_0)<+\infty}=1$.
\end{Proposition}

\begin{proof}
{\it Step 1.}
Consider $R_{\eps}:=R_5 \cup R_6 \cap \{x \le \wt x - \eps \} $
for arbitrarily fixed $\eps>0$.
Assume that $(x_0,y_0)\in R_{\eps}$.
Denoting $\vp_t:=\vp(x_t, y_t)\ge 0$ with $\vp~:~(x,y)\in D\mapsto 1/y$
and recalling Definition \ref{def:theta},
$
\Lc \vp(x,y) = -\vp(x,y) (\kappa(x)-\gamma)<-(\kappa(\wt x-\eps)-\gamma)/f(0)<0
$
for very $(x,y)\in R_{\eps}$.
Theorem \ref{th:reccurent domain} then states
that
$(x_t,y_t)$
exits $R_{\eps}$ in 
finite time almost surely.
Since that
$\theta$ is non-decreasing on this set,
and recalling
Theorem \ref{th:existence},
it is only possible via $R_7$ and
$
\Pro{\tau_7(x_0, y_0)<+\infty}=1$.
This holds for any $\eps>0$.

{\it Step 2.}
Assume now that $(x_0, y_0)\in (R_5\cup R_6) \backslash R_{\eps}$.
According to step 1,
$\{\tau_7(x_0,y_0)=+\infty\}\subset \{x_t\ge \wt x,\ \forall t\ge 0\}$ 
and thus
$\{\tau_7 (x_0,y_0)=+\infty\}\subset \{\theta_t\le f(\wt x)/\wt x, \ \forall t\ge 0 \}$.
Because $\theta_t$ is non decreasing and
according to
Doob's martingale convergence theorem,
$\theta_t$ 
converges to 
$\theta_0 \in L^\infty([\we \theta, f(\wt \omega)/\wt \omega])$
on $\{\tau_7 (x_0,y_0)=+\infty\}$.
This implies that $(x_t,y_t)$
converges with $t$ to $R_6\cap R_7$
on $\{\tau_7 (x_0,y_0)=+\infty\}$.
Since $\sigma(y)> \sigma(f(0))$, 
this convergence is improbable.
 \end{proof}

\subsection{Northern quadrant}
\label{sec:north}

Finally we prove that if $(x_0,y_0)\in R_7$,
then the process reaches $R_1$ in finite time almost surely.
One can notice that proofs are very similar to
those of Subsections \ref{sec:west} and \ref{sec:south}.

\begin{Proposition}
\label{prop:leave_R6_union_R7}
If $(x_0, y_0) \in R_6\cup R_7$ then 
$\Pro{\tau_5  \wedge \tau_8(x_0, y_0)<+\infty}=1$.
\end{Proposition}

\begin{proof}
Define the sequence of regions $\{B_n\}_{n\in\N}$ through	
$
	B_n = R_6 \cup R_7 \cap\{y<1-k/n\} \cap \{x>k/n\}
$
where $k>0$ is sufficiently small to have
$(x_0,y_0) \in B_1$. 
Applying It\^o to $\vp ~:~(x,y)\in D \mapsto \sqrt{\wt x-x_t}$
%
we find that
for all $(x,y)\in B_n$

\begin{align*}
	\Lc \vp(x,y) 
	&=& -\frac{1}{2\vp(x,y)}\l[x\l[\Phi(y)-\alpha+\sigma^2(y )\r]+\frac{1}{4}\frac{x^2}{\wt x-x}\sigma^2(y )\r] 
	\le -\frac{x^2 \sigma^2(y)}{8(\wt x-x)^{3/2}}\\
	& & \le -\frac{k^2 \sigma^2(1-k/n)}{8\sqrt{n}(n\wt x-k)^{3/2}} <0 
\end{align*}
while 
$\Lc \vp(x,y)\le 0$ in $R_6\cup R_7$.
Doob's supermartingale convergence theorem implies 
the existence of the pointwise limit
$
\vp_{\infty}:=	\lim_t \vp(x_{t  \wedge \tau_{5,8}},y_{t  \wedge \tau_{5,8}})
$
almost surely,
where we use the notation 
$\tau_{5,8} := \tau_5(x_0,y_0)  \wedge \tau_8(x_0, y_0)$. 
In addition, 
Theorem \ref{th:reccurent domain} guarantees 
that every set $B_n$ is exited in finite time almost surely.
Consequently 
if $\omega \in \{\tau_{5,8} = +\infty\}$, 
we have that either 
$\lim_t x_t(\omega)= 0$ or 
$\lim_t y_t(\omega)= 1$, a contradiction in either way. 
 \end{proof}

\begin{Proposition}
\label{prop:R7_R8}
If $(x_0, y_0)\in R_6\cap R_7$ then
$\Pro{\tau_8(x_0, y_0)<+\infty}=1$.
\end{Proposition}

\begin{proof}
The proof is identical to the one of
Proposition \ref{prop:R2_to_R4},
with small modifications.
Here
$\vp_t := \vp(x_t, y_t)$ with
$\vp~:~(x,y)\in D\backslash\{(\we x, \we y)\} \mapsto (y - \we y)/(x-\we x)$
and
$F~:~(x,y)\in D\backslash\{(\we x, \we y)\} \mapsto \tan\l( \tan^{-1}(\vp(x,y)) + \tan^{-1}(c) \r)
$.
The process
$F_t := F(x_t, y_t)$
is a supermartingale on $R_5 \cup R_6 \cup R_7$
if we chose $c\in (0, (\we \theta + M/m)^{-1})$
where $(m,M)$ are two positive constants given by
$
	m:= \min_{[\we x, \wt x] \times [\wt y, \we y]} x \wt x \sigma^2(y)
$
and
$
	M:=\max_{[\we x, \wt x] \times [\wt y, \we y]} y(x-\wt x)
	\l[\kappa(x)-\gamma\r]-x(y-\wt y)\l[\Phi(y)-\alpha\r] \; .
$
The justification is the following.
The domain 
$S_c := D \backslash \{\we \theta \le \vp(x,y)\le 1/c \}$
contains the area of interest $R_5 \cup R_6 \cup R_7$.
Using Proposition \ref{prop:R2_to_R4},
we can prove that $F_t$ is a supermartingale
on $S_c \backslash [\wt x, \we x]\x [\we y, \wt y]$.
On $[\wt x, \we x]\x [\we y, \wt y]$,

\begin{align*}
	(x-\we x)^2\frac{(1-R_t c)^2}{1+c^2}\Lc F_t \le
	&  \ y (x - \we x)[\kappa(x)-\gamma] - x(y - \we y)[\Phi(y)-\alpha]\\
	& +\sigma^2(y)\left[y(x-\we x) - x (y - \we y)  +	\we x \left(y - x(1/c - \we x + \we x)\right)\right]\\
	\le &  \ y (x - \we x)[\kappa(x)-\gamma] - x(y - \we y)[\Phi(y)-\alpha] - x\we x \sigma^2(y)(1/c - \we \theta)\\
	\le & \ M - m(1/c - \we \theta)\le 0 \;.
\end{align*}
We then reproduce step 2 of the proof of Proposition \ref{prop:R2_to_R4},
using Propositions \ref{prop:R5_to_R7} and \ref{prop:leave_R6_union_R7}
instead of Propositions \ref{prop:R1_R2} and \ref{prop:leave R_2 and R_3}.
 \end{proof}

\begin{Proposition}
\label{prop:R8_R1}
If $(x_0,y_0)\in R_7\cap R_8$ then
$\Pro{\tau_1(x_0,y_0)<+\infty}=1$.
\end{Proposition}

\begin{proof}
We follow Proposition \ref{prop:R4_to_R5}
with the minor following modifications.

{\bf 1}
We consider 
$\tau_{1,6}:=\tau_1(x_0, y_0)\wedge \tau_6(x_0, y_0)$
the exit time of $R_7 \cup R_8$.
The process $\vp_t:=\vp(x_t, y_t)$
with $\vp ~:~(x,y)\in D \mapsto x_t^{-2}$
verifies

$$
\Lc h_t = - 2 h_t \left(\Phi(y_t)-\alpha + \frac{3}{2}\sigma^2(y_t)\right)< -\eps h_t < -\eps h(\we \theta)<0
$$
for some $\eps >0$.
Indeed $\Phi(y)-\alpha+\sigma^2(y)\ge 0$
and is null only if $y=\we y$,
whereas $\sigma^2(y)=0$ only if $y=1$.
Applying Theorem \ref{th:reccurent domain}
to $R_7\cup R_8$,
$\tau_{1,6}< + \infty$ almost surely.

{\it Step 2.}
If $(x_{\tau_{1,6}}, y_{\tau_{1,6}})\in R_6$,
then the process reaches $R_8$
in finite time almost surely according to
Proposition \ref{prop:R7_R8}.
We define the sequence
$( \tau_{1,6}^n, \tau_8^n)_{n\ge 0}$
with $\tau_8^0:=0$ and

$$
\l\{
\begin{array}{ll}
\tau_{1,6}^n :=& \inf\{t \ge \tau_8^{n} ~:~ (x_t, y_t)\in R_6\cup R_1 \}\\
\tau_8^{n+1} :=& \inf\{t>\tau_{1,6}^{n} ~:~ (x_t,y_t) \in (R_7\cap R_8)\cup R_1 \}
\end{array}
\r. , \ \pourtout n\ge 0\;.
$$
Proceeding as in step 2 Proposition \ref{prop:R4_to_R5},
we obtain that $\Pro{\tau_1(x_0, y_0)=+\infty}>0$
implies that

\be 
\label{eq:limite proba}
\lim_n \Pro{x_{\tau_{1,6}^n}<\we x | x_{\tau_8^n} = \we x}=1 \; .
\ee

{\it Step 3.}
Define $m:= \inf\{2(\Phi(y)-\alpha)+3\sigma^2(y) ~:~y \in [\we y, 1) \}$,
which is strictly positive
according to step 1.
Consider the new
process $\vp_t:= \vp(x_t,y_t)$
with $\vp~:~(x,y)\in D \mapsto \exp(-mt)x^2_t$.
It is a positive submartingale on $[0, \tau_{1,6}^0]$,
and similarly to step 3 of Proposition \ref{prop:R4_to_R5},
we can obtain

\begin{align*}
\we x^2
&\le& \Esp{x^2_{\tau_{1,6}^n} e^{-m(\tau_{1,6}^n - \tau_8^n)} | x_{\tau_8^n} = \we x}
 \le \we x^2 \Esp{e^{-m(\tau_{1,6}^n - \tau_8^n)}\ind{x_{\tau_{1,6}^n}<\we x} | x_{\tau_8^n} = \we x} \\
& & + \we \theta^{-2}\Esp{e^{-m(\tau_{1,6}^n - \tau_8^n)}\ind{x_{\tau_{1,6}^n}\ge \we x} | x_{\tau_8^n} = \we x}\; .
\end{align*}
Assuming \eqref{eq:limite proba},
we have

\begin{align*}
0\le \Esp{(1-e^{-m(\tau_{1,6}^n - \tau_8^n)})\ind{x_{\tau_{1,6}^n}<\we x} |x_{\tau_8^n} = \we x }
\le (1/\we \lambda - 1) (1- \Pro{x_{\tau_{1,6}^n}<\we x | x_{\tau_8^n} = \we x}) \xrightarrow{n} 0\; .
\end{align*}
We then proceed exactly as in
step 3 of Proposition \ref{prop:R4_to_R5} to finish the proof.
 \end{proof}

\section{Example}
\label{sec:example}

In this section we assume that
investment follows Say's law
and Philips curve 
is provided by \cite{grasselli2012, keen1995}.

\begin{Assumption}
 \label{ass:form of functions}
 We let
 $\kappa ~:~x\in \R_+ \mapsto (1-x)/\nu$ 
 and 
 $\Phi~:~y\in [0,1)\mapsto \Phi(y):= \frac{\phi_1}{(1-y)^2}+\phi_0 $.
\end{Assumption}

Assumption \ref{ass:non-linear goodwin} holds under Assumption \ref{ass:form of functions}.
The unique non-hyperbolic equlibrium point in $D= \R^*_+\x (0,1)$ is given by
$
(\wt x, \wt y) = \left(1-\nu \gamma, 1 - \sqrt{\phi_1/(\alpha - \phi_0)}\right) \; .
$
Functions of Definition \ref{def:Lyapunov} 
are given by $V(x,y)=V_1(x)+V_2(y)$ with

\begin{equation}
\label{eq:V_1}
\begin{array}{lll}
V_1(x) &=& 
\frac{1}{\nu}\l(x - \wt x \l( \log\l(\frac{x}{\wt x}\r)+1 \r) \r)\;, 
\\
V_2(y)&=&
\phi_1 \l( \log \l(\frac{1-\wt y}{1-y}\r)
+\l(\frac{\wt y}{1-\wt y}\r)
\log \l(\frac{\wt y}{y}\r)
+\frac{1}{y - y^2}- \frac{1}{\wt y - \wt y^2} \r)\; .
\end{array}
\end{equation}
Although
period $T$ given by 
Theorem \ref{th:periods_goodwin}
is not explicit here,
numerical computations allow
to approximate it with a linear
function of $V_0$, see first part of Fig. \ref{fig:linear_rel}.
Following
Remark \ref{rem:period},
$T$ does not converge to $0$
with solutions of \eqref{eq:goodwin_model} 
concentrating to $(\wt x, \wt y)$.
A local phase portrait with values of $T$
is provided in second part of Fig. \ref{fig:linear_rel}.
\begin{figure}[!ht]
\centering
\includegraphics[width= \textwidth, height=11cm]{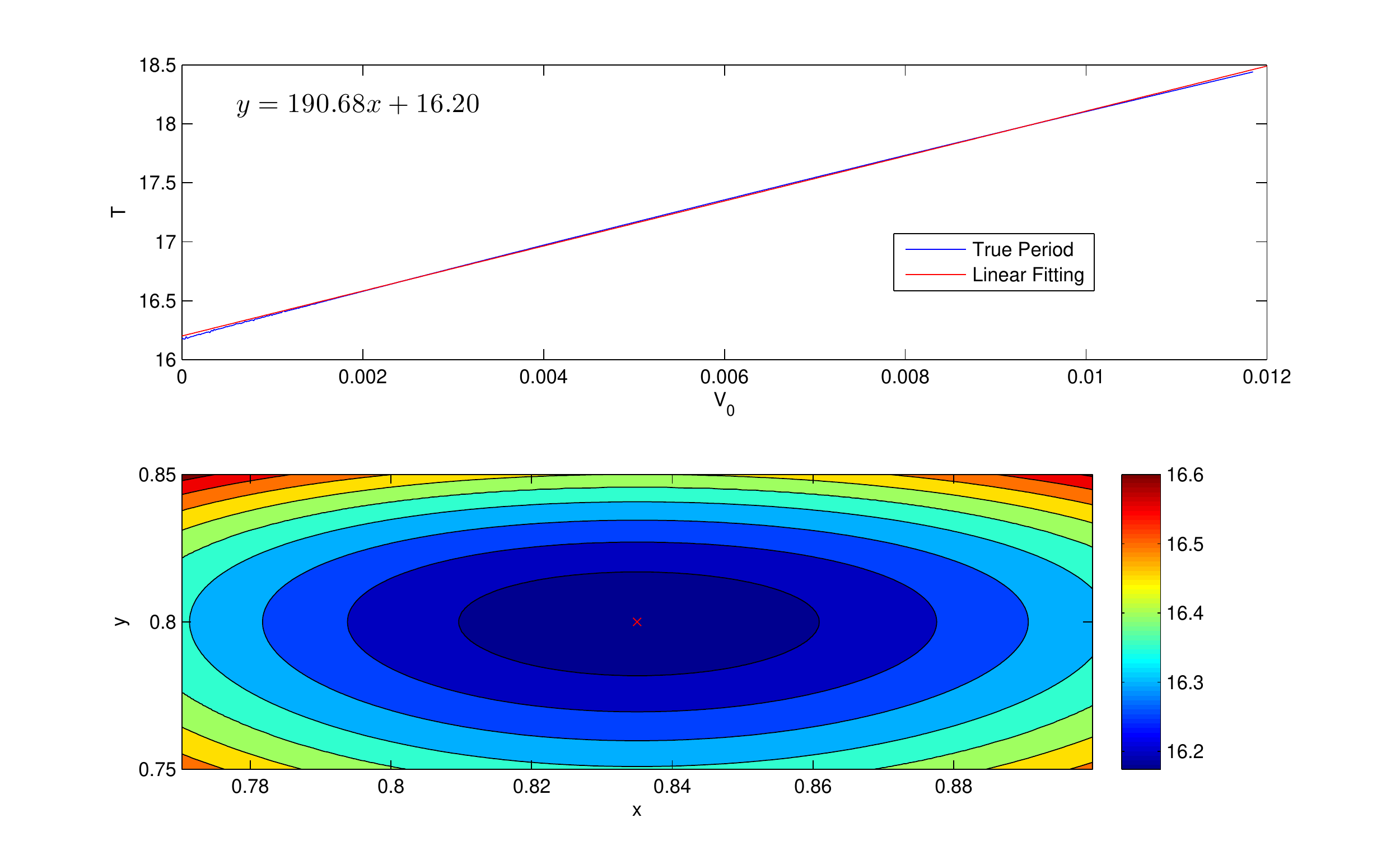}
\caption{Up: values of $T$ as a function of $V$.
Down: Contour lines with values of $T$ in a subset of $D$.
Parameters set at
$(\alpha, \gamma, \nu, \phi_0, \phi_1)
=(0.025, 0.055, 3, 0.040064,0.000064)$. 
Equilibrium point at
$(\wt x, \wt y) = (0.8350, 0.80)$.}
\label{fig:linear_rel}
\end{figure}

\begin{Assumption}
 \label{ass:form of vol}
 Let $\sigma~:~y\in [0,1] \mapsto \sigma_0 (1-y)$ with $\sigma_0>0$.
\end{Assumption}

If we implicitly assume that the
perturbation of the average growth rate $\alpha$ of the productivity
is due to the flow of workers coming in and out of
the fraction $y_t$ employed at time $t$,
Assumption \ref{ass:form of vol} conveniently expresses
that this perturbation decreases with the employment rate
since higher employment implies lower perturbation on
the constant average rate $\alpha_t$.
Other models can of course be considered.

Assumption \ref{ass:form of vol}
together with Assumption \ref{ass:form of functions},
and comparing with \eqref{eq:V_1},
satisfy Assumption \ref{ass:growth conditions}.
Indeed for all $y\in (0,1)$,

\begin{equation}
\label{eq:estimate_y}
\sigma^2(y)\Phi'(y) = \frac{2\sigma_0^2 \phi_1}{(1-y)}
 \le 2 \sigma_0^2 \left(V_2(y) -\phi_1\l(\frac{1}{1-\wt y}\l(\wt y \log(\wt y)-\frac{1}{\wt y} \r)+\log(1-\wt y) \r) \right)
\end{equation}
and along with the sub-linearity of the log function,

\begin{equation}
\label{eq:estimate_x}
-\kappa(x)-x\kappa'(x)=\frac{2x-1}{\nu} \le \frac{2}{1-\wt x}\l(V_1(x) +\wt x - \wt x \log(\wt x)  \r)\;.
\end{equation}
In line with
Assumptions 
\ref{ass:non-linear goodwin} and \ref{ass:growth conditions}
the vertical asymptote at $y=1$
implies that
$
\sigma^2(y)\Phi(y)\le K_0  V_2 (y) +k_0
$
for some $K_0,k_0\in \R_+^2$.
Under Assumption \ref{ass:growth conditions} 
and following \eqref{eq:estimate_y} and \eqref{eq:estimate_x}, 
$K_0=0$ and $k_0=\sigma_0^2 (\phi_1+\phi_0)$.

Assumption \ref{ass:form of vol} also
implies that $(1-\we y)^2$ is the root of a quadratic polynomial
$
	\sigma_0^2 (1-\we y)^4 - (\alpha+\phi_0) (1-\we y)^2 + \phi_1 = 0 \;.
$
The latter shall have a unique root in $(0,1)$
to satisfy Assumption \ref{ass:unique_root}.
The following example of condition is sufficient.
\begin{Assumption}
 \label{ass:unique root}
 We assume $\phi_1 \le (\alpha + \phi_0)/2$ and
$\sigma_0 \le \max \{(\alpha+ \phi_0)/(2\sqrt{\phi_1}), (\alpha + \phi_0- \phi_1) \}$.
\end{Assumption}

We are now
able to claim
the existence of
$K,k\in \R^2_+$ such that
$|R(V_0,\rho)| \le K(V_0+\rho)+ k $,
where $R$ is defined in Proposition \ref{prop:estimate}.
A direct application provides

$$
	R(V_0, \rho) := \max\limits_{D(V_0,\rho)} \l\{
	\sigma_0^2(1-y)^2\l( \frac{2x-\wt x}{\nu} + \frac{2 \phi_1 y}{(1-y)^3}
	 + \frac{\phi_1}{(1-y)^2}-\frac{\phi_1}{(1-\wt y)^2}
	\r)\r\}.
$$
Using \eqref{eq:estimate_x},
this estimate becomes
$
|R(V_0, \rho) |\le K (V_0+\rho) + k$
where $K:=(2\sigma_0^2)/(1-\wt x)$ 
and $k$ is an 
explicitly calculable constant.
Following the same procedure with \eqref{eq:estimate_y},
$I(V_0, \rho)\le K (V_0+\rho)^2+k'$
with the same $K$ and $k'\ne k$.
Now choosing 
$\mu=(\rho-(K(V_0+\rho)+k)\theta/2)/ \sqrt{\theta}$
for some $\theta\ge 0$,
so that $\Theta(\rho) = \theta$,
Proposition \ref{prop:estimate} provides

$$
	\Pro{\tau_\rho > \theta} 
	\ge \l(1-\frac{(K (V_0+\rho)^2 + k') \theta}
	{\l(  \demi(K(V_0+\rho)+k)\theta - \rho\r)^2}\r)\; .
$$

Theorem \ref{th:finite period}
is a straightly observable phenomenon
with simulations, see Fig. \ref{fig:panel_goodwin_stochastic_examples}.
Under the assumptions of this section,
the system has been simulated
using XPPAUT with a 
fourth order Runge-Kutta scheme for 
the deterministic part,
and an Euler scheme for the Brownian part.
Fig. \ref{fig:panel_goodwin_stochastic_examples}
illustrates the effect of the 
volatility level $\sigma_0$
on trajectories of the system,
as for the economic quantity $P_t:=a_t y_t N_t$.
\begin{figure}[!ht]
\centering
\includegraphics[width=0.9\textwidth, height=11cm]{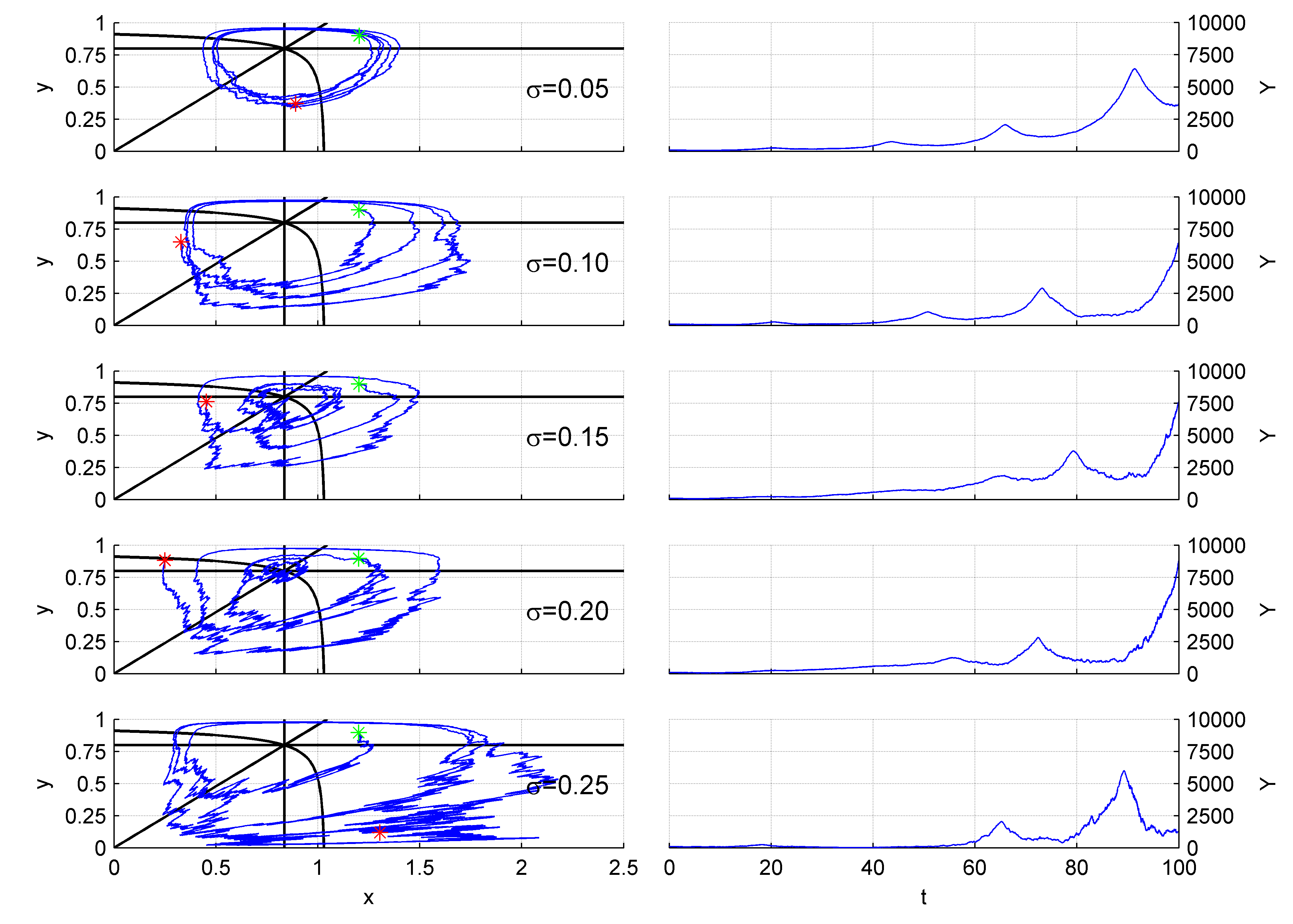}
\caption{Left column : phase diagram $(x,y)$
of subsample paths of trajectories for \eqref{eq:stochastic_goodwin_1}
with different values of volatility $\sigma_0$,
starting from the green star and stopping at the red start.
Right column: evolution of output $P_t$ over time for the subsample path.}
\label{fig:panel_goodwin_stochastic_examples}
\end{figure}

Apart from specific subregions of $D$
as $R_1$ or $R_5$ where
Corollary 3.2 in \cite{khasminskii} can provide
an estimate for the expectation of 
the exit time, 
a bound for
the expected period $\Esp{S}$
seems out of reach.
Numerical simulations have nevertheless
always provide reasonnable finite
periods of stochastic orbits of \eqref{eq:stochastic_goodwin_1}.
We thus expect that $\Esp{S}$ is finite for
a wide range of values of $(x_0, y_0)\in D$.
Let us start with $(x_0,y_0)\in R_1\cap R_8$
and reformulate
$S$ of Definition \ref{def:stochastic orbit}
as the time the process crosses the line $y = \wt\theta x$
for the second time.
This is equivalent to take $(x_0,y_0)\in R_4\cap R_5$.
Resorting to numerical methods, 
we have simulated the system $2000$ times for $100$ different starting points 
in $R_1\cap R_8$
and recorded the position at 
the time when this line is crossed the second time, 
that is the positions after a full loop. 
Fig. \ref{fig:goodwin_stochastic_expected_loop} 
contains such examination for an array of values of $\sigma_0$.
The expected time $\Esp{S}$ 
to complete a full-loop 
is also illustrated. 
As observed, 
there seems to be a stable 
attractive fixed point to
$y_0 \mapsto \Esp{y_S}$ for sufficiently large values of $\sigma_0$.
If the starting point is picked 
too close to $(\wt x, \wt y)$, 
the expected crossing value after one loop 
is further away from it. 
On the other hand, 
if the one starts extremely far away 
from $(\wt x, \wt y)$, 
say with $y_0<0.25$, 
then the expected value after on loop is higher. 
This implies that 
after many loops,
the expectation converges, and so
does $\Esp{S}$ with the number of loops
around $(\we x, \we y)$.
Assuming that $\Esp{S}<+\infty$ for enough initial points,
Theorem \ref{th:reccurent domain} can be used
with $V$ 
at points $(0,0)$ and $(\we x, \we y)$
to prove the following conjecture.
\begin{Conjecture}
\label{prop:fixed_point}
Consider the function $\Sc~:~ y\in (0,\we y)\mapsto \Esp{y_S}\in (0,\we y)$ such that
$(x_t,y_t)$ is a solution to \eqref{eq:stochastic_goodwin_1}
with $(x_0, y_0) = (y/\we \theta, y)$, 
and $S$ is the finite stopping time defined by Theorem \ref{th:finite period}.
Then $\Sc$ has at least one fixed point in $(0,\we y)$.
\end{Conjecture}

\begin{figure}[!ht]
\includegraphics[width=\textwidth, height=11cm]{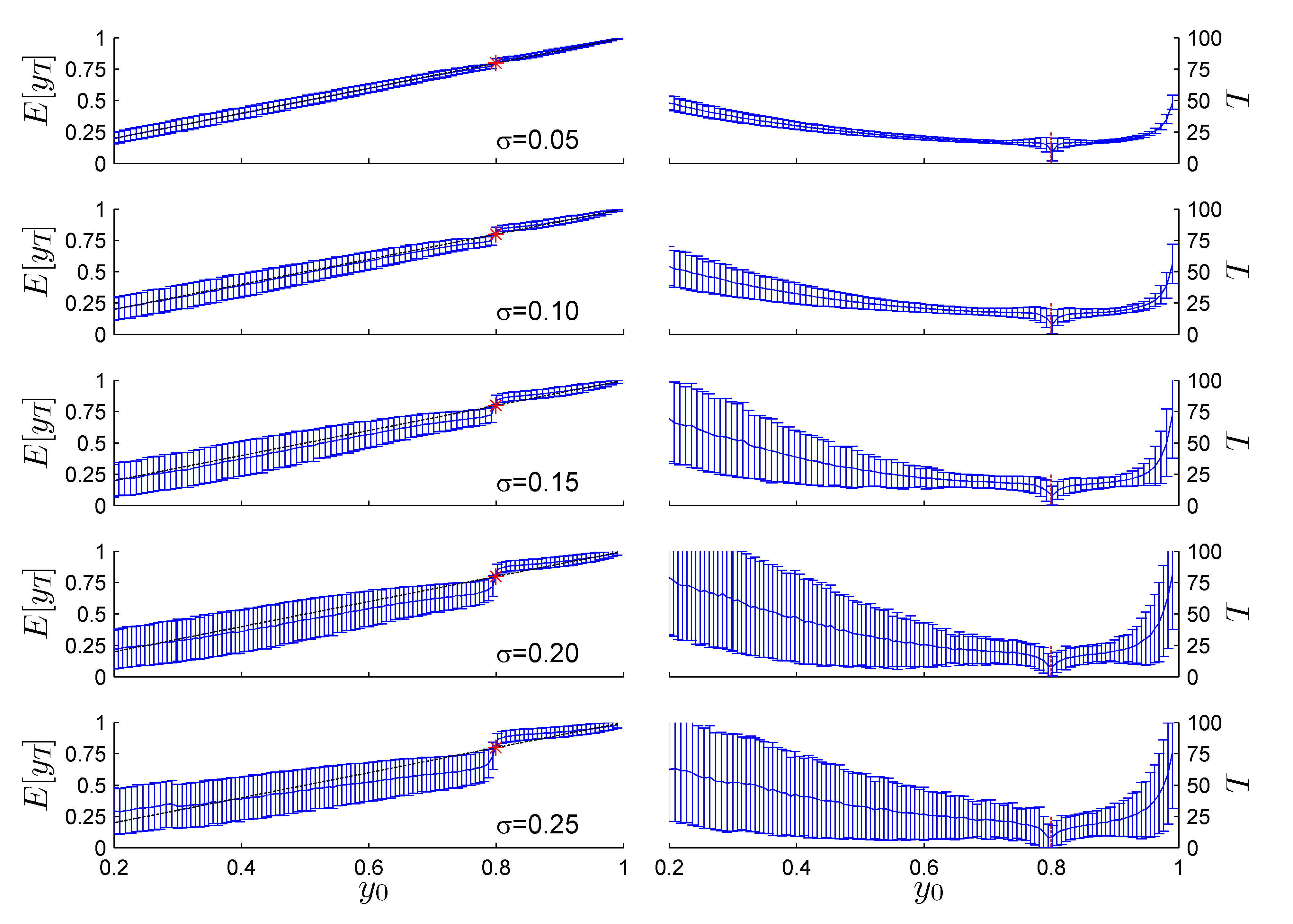}
\caption{Expected values of employment $y$
after one full loop $y_T$ (left), and expected elapsed time $T$ (right). 
Computation performed in MATLAB,
with $2000$ simulations for 
every value single one of 
the $100$ initial values taken along 
the line $y = \wt\theta x$.}
\label{fig:goodwin_stochastic_expected_loop}
\end{figure}

\section{Concluding remarks}

This contribution attempts to
draw the attention of dynamical system analysis
onto macroeconomic models.
Before looking into complex models
of finance and crises, e.g. \cite{costa2014, grasselli2012, keen1995},
we focus here on a Brownian perturbation added into
a non-linear version of the Lotka-Volterra system
used in Economics, the Goodwin model.
To begin with,
we recall 
the usual results for the deterministic
planar oscillator: we provide 
the constant Entropy function and
describe the period of 
the closed orbits drawned by the system.
We then provide
sufficient conditions for the 
stochastically perturbed system
to stay in the meaningful domain
$D$ which is a a bounded subset of
$\R^2_+$ for the $y$-component.
The entropy function is actually of
great use for the last result, 
additionally to prior estimates on
variations of the system.

We finally prove what seems a
fundamental and staightforward property of the system,
namely that a solution $(x_t, y_t)$ rotates
with perturbations around a unique point $(\we x, \we y)$.
The definition of stochastic orbits provided here
conventienly suits the intuition of
how the deterministic concept can be extended.
However it has clearly not
the ambition to be a definitive concept and further
investigations might confirm its usefulness
or its precarity.
The proof exploits the concept of reccurent domains
in an intensive manner.

We expect that economists
seek interest in \eqref{eq:stochastic_goodwin_1},
as other perturbed macroeconomic systems (e.g. \cite{kiernan1989, neamtu2012}),
for the possibility to adjust the model
to observed past data (e.g. \cite{arato2003} and \cite{harvie2000, mohun2006}) and find a possible
synthetic explanation for perturbations of business cycles (see \cite{evans1992, hansen1985}).

\section*{Acknowledgment}
Both authors want to thank
Matheus Grasselli for
leading ideas and presentation
suggestion.
Remaining errors are authors responsibility.



\end{document}